\documentclass[12pt,reqno]{amsart}
\usepackage[all]{xy}
\usepackage{amsfonts,amsmath,amssymb,amscd}
\usepackage{color}
\usepackage[pagebackref,colorlinks, linkcolor=blue, citecolor=red]{hyperref}
\usepackage{caption} 
\usepackage{mathtools}

\usepackage[top=1in, bottom=1in, left=1in, right=1in]{geometry}
\usepackage{graphicx}

\def\opn#1#2{\def#1{\operatorname{#2}}} 
\opn{\cone}{cone}
\opn{\PF}{PF}
\opn{\F}{F}
\opn{\RF}{RF}
\opn{\Ap}{Ap}

\numberwithin{equation}{section}

\newtheorem{theorem}{Theorem}[section]
\newtheorem{proposition}[theorem]{Proposition}

\newtheorem{lemma}[theorem]{Lemma}
\newtheorem{coro}[theorem]{Corollary}
\newtheorem{definition}[theorem]{Definition}

\newtheorem{remark}[theorem]{Remark}
\newtheorem{example}[theorem]{Example}

\begin{document}
	
\title{On Row-Factorization relations of certain numerical semigroups} 

\address{IIT Gandhinagar, Palaj, Gandhinagar, Gujarat-382355 India}

\author{Om Prakash Bhardwaj}
\email{om.prakash@iitgn.ac.in}

\author{Kriti Goel}
\email{kritigoel.maths@gmail.com}

\author{Indranath Sengupta}
\email{indranathsg@iitgn.ac.in} 

\thanks{The second author is supported by the Early Career Fellowship at IIT Gandhinagar.}
\thanks{The third author is the corresponding author; supported by the MATRICS research grant MTR/2018/000420, sponsored by the SERB, Government of India.}

\maketitle

\begin{abstract}
	Let $H$ be a numerical semigroup minimally generated by an almost arithmetic sequence. We give a description of a possible row-factorization $(\RF)$ matrix for each pseudo-Frobenius element of $H.$ Further, when $H$ is symmetric and has embedding dimension 4 or 5, we prove that the defining ideal is minimally generated by $\RF$-relations.
\end{abstract}

\section{Introduction}

Numerical semigroups and semigroup rings are widely studied objects in the literature (see \cite{rosales-sanchez}). Recently, A. Moscariello \cite{moscariello} introduced a new tool associated with the pseudo-Frobenius elements of numerical semigroups, namely, row-factorization ($\RF$) matrices, to explore the properties of semigroup rings. Moscariello used this object to investigate the type of almost symmetric semigroups of embedding dimension four and proved the conjecture given by T. Numata in \cite{numata}, which states that the type of an almost symmetric semigroup of embedding dimension four is at most three. The connection between the $\RF$-matrices and numerical semigroup rings is further studied by K. Eto, J. Herzog and K.-i. Watanabe in (\cite{etoGeneric}, \cite{etoRowFactor}, \cite{eto2017}, \cite{herzogWatanabe}), where they explore the set theoretic complete intersection property, minimal free resolution, generic property, etc for the defining ideal of certain monomial curves. In \cite{etoGeneric}, K. Eto gives necessary and sufficient conditions for a toric ideal to be generic in terms of $\RF$-matrices and as a result proves that the defining ideal of almost Gorenstein monomial curves is not generic if their embedding dimension is greater than three.

Let $k$ be a field and $H = \langle m_0,m_1,\ldots,m_p \rangle$ be a numerical semigroup minimally generated by $\{m_0,m_1,\ldots,m_p\}$. For an indeterminate $t$ over $k$, the semigroup ring of $H$ is a subring of $k[t]$, defined as $k[H] = k[t^{m_0},\ldots,t^{m_p}].$ The ring $k[H]$ can be represented as a quotient of a polynomial ring using a canonical surjection $\pi : k[x_0,x_1,\ldots,x_p] \rightarrow k[H]$ given by $\pi(x_i) = t^{m_i},$ for all $i.$ The kernel of this $k$-algebra homomorphism $\pi$, say $I(H),$ is a prime ideal (called toric ideal) and is the defining ideal of $k[H]$. The ring $k[H]$ is called a toric ring. 

In this paper, we consider the numerical semigroups generated by an almost arithmetic sequence. In particular, we have $H = \langle m_0,m_1,\ldots,m_p,n \rangle$ such that $m_i = m_0 + id $ for $i = 0,\ldots,p$, $d >0$ and $\gcd(m_0,n,d) = 1.$ The aim of this paper is to prove that if $H$ is a symmetric numerical semigroup generated by an almost arithmetic sequence and has embedding dimension 4 or 5, then the defining ideal is minimally generated by $\RF$-relations. 

This, in the first place, requires knowledge of the pseudo-Frobenius elements of $H.$ For this, we refer to the computations in \cite{patil-sengupta}, where the authors give an explicit description of pseudo-Frobenious elements for the numerical semigroups of this class. Also, for a numerical semigroup $H$, there is a one-to-one correspondence between the set of pseudo-Frobenius elements and the set of maximal elements of the A\'{p}ery set of a non-zero element of $H$ with respect to the partial order $\leq_H$ defined by: $x \leq_H y$ if and only if $y-x \in H$ (see \cite[Proposition 2.20]{rosales-sanchez}).  A description of the A\'{p}ery set,  Frobenius number of a numerical semigroup generated by an almost arithmetic sequence is given in \cite{RamirezRodseth} and \cite{rodsethoystein}.

We now summarize the contents of the paper. In Section 2, we recall some basic definitions, along with the definition of $\RF$-matrices. We also recall some results from the papers \cite{patil-sengupta} and \cite{patil-singh}, which are used in the further sections. In Sections 3 and 4, we give the computation of $\RF$-matrices. A numerical semigroup $H$ of type 1 is called symmetric. It is known that $k[H]$ is Gorenstein if and only if $H$ is a symmetric numerical semigroup. In section 5, we identify the cases in which a numerical semigroup $H$, minimally generated by an almost arithmetic sequence, is symmetric. We compute the $\RF$-matrices in each of these cases. The rows of $\RF$-matrices are known to produce binomials in the ideal $I(H).$ These binomials are called $\RF$-relations. Since $I(H)$ may not always be minimally generated by $\RF$-relations, one may ask when is such a scenario possible. This question was raised by Herzog and Watanabe in \cite{herzogWatanabe}. They answered the question in the affirmative for the cases when the embedding dimension is 3 and when $H$ is pseudo-symmetric or almost symmetric with embedding dimension 4. In Theorem \ref{p=2,t=1} and Theorem \ref{p=3,t=1}, we prove that $I(H)$ is minimally generated by $\RF$-relations when $H$ is a symmetric numerical semigroup, minimally generated by an almost arithmetic sequence with embedding dimension 4 or 5.

\section{Notation and Preliminaries}

Let $\mathbb{Z}$ and $\mathbb{N}$ be the sets of integers and non-negative integers respectively. Let $p>0$ and $m_0,m_1,\ldots,m_p \in \mathbb{N}$ such that $\gcd(m_0,m_1,\ldots,m_p) = 1.$ Then
\[ \langle m_0,m_1,\ldots,m_p \rangle = \left\lbrace \sum_{i=0}^{p} a_i m_i \mid a_i \in \mathbb{N}, \forall i \right\rbrace\]
is called the numerical semigroup generated by $m_0,m_1,\ldots,m_p.$


\begin{definition}{\rm 
Let $H$ be a numerical semigroup and $m$ be a non-zero element of $H.$ The set $\Ap(H,m) = \{h\in H \mid h-m \notin H \}$ is called the Ap\'ery set of $H$ with respect to $m.$
}
\end{definition}

\begin{definition}
{\rm Let $H$ be a numerical semigroup. An element $f \in \mathbb{Z} \setminus H $ is called a pseudo-Frobenius 
number if $f + h \in H$ for all $h \in H \setminus \{0\}$. We will denote the set of pseudo-Frobenius numbers 
of $H$ by $\PF(H).$
}
\end{definition}

Note that if $f \in \PF(H)$, then $f + m_i \in H$ for all $i \in [0,p].$ Hence, for all $i \in [0,p],$ there exist $a_{i0}, a_{i1},\ldots, a_{ip} \in \mathbb{N}$ such that
\[
f + m_i = \sum_{j=0}^{p} a_{ij} m_j.
\]
In the above expression, $a_{ii} = 0.$ Because if $a_{ii} > 0,$ then it would imply $f \in H.$ We now recall the notion of row-factorization matrix ($\RF$-matrix) given by A. Moscariello in \cite{moscariello}.

\begin{definition}
{\rm Let $H$ be a numerical semigroup generated by $m_0,\ldots,m_p$ and $f \in \PF(H)$. We say that $ A = (a_{ij}) \in M_{p+1} (\mathbb{Z})$ is an $\RF$-matrix for $f$ if for all $i \in [0,p],$
	\[
 	\sum_{j=0}^{p} a_{ij} m_j = f,
	\]
	where $a_{ii} = -1$ for all $i \in [0,p]$, and $a_{ij} \in \mathbb{N}$ for all $i, j \in [0,p]$ and $i \neq j.$ For $f \in \PF(H)$, we will denote an $\RF$-matrix of $f$ by $\RF(f)$.
	}
\end{definition}

Observe that for a given $f \in \PF(H),$ $\RF(f)$ may not be not unique. A factorization of $f + m_i$ gives the $(i+1)^{th}$ row of an $\RF$-matrix for $f.$ Thus, if $f+m_i$ does not have unique factorization in $H$, then the matrix $\RF(f)$ is not unique. Nevertheless, $\RF(f)$ will be the notation for one of the possible $\RF$-matrices of $f.$ If $H$ is symmetric, i.e. $\text{type}(H) = |\PF(H)|=1,$ then $\RF(H)$ denotes an $\RF$-matrix of the only pseudo-Frobenius number of $H.$

Assume that the sequence $m_0,m_1,\ldots,m_p$ is a strictly increasing arithmetic sequence of positive integers and $n$ is a positive integer such that $\gcd(m_0,\ldots,m_{p},n) = 1.$ Also, assume that $\{m_0,\ldots,m_p,n\}$ is a minimal generating set for the numerical semigroup $H.$ The following result was proved by D.P. Patil and B. Singh in \cite{patil-singh}.

\begin{lemma}[{\cite[Lemma 3.1, 3.2]{patil-singh}}]  \label{P-I}
	Let $d$ be a positive integer such that $m_i=m_0+id$ for all $0 \leq i \leq p.$ Let $n$ be an arbitrary positive integer with $\gcd(m_0,d,n) = 1.$ Let $H' = \sum_{i=0}^{p} \mathbb{N} m_i$ and $H = H'+\mathbb{N}n.$ For $t \in \mathbb{N},$ let $q_t \in \mathbb{Z}, r_t \in [1,p]$ and $g_t \in H'$ defined by $t=q_tp + r_t$ and $g_t = q_tm_p + m_{r_t}.$
	\begin{enumerate}
		\item Let $u = \min \{t\in \mathbb{N} \mid g_t \notin \Ap(H,m_0) \}$ and $v = \min \{ b \geq 1 \mid bn \in H'\}.$ Then there exist unique integers $z \in [0,u-1], w \in [0,v-1], \lambda \geq 1$, and $\mu \geq 0$ such that
		\begin{itemize}
			\item[(i)] $g_u = \lambda m_0 + wn;$
			\item[(ii)] $vn = \mu m_0 + g_z;$ and	
			\item[(iii)] $g_{u-z} + (v-w)n = \nu m_0.$ Moreover, $\nu = \lambda + \mu + \epsilon$ where $\epsilon=1$ or $0$ according as $r_{u-z} < r_u$ or $r_{u-z} \geq r_u.$
		\end{itemize}
		
		\item $g_s + g_t = \epsilon m_0 + g_{s+t}$ with $\epsilon = 1$ or $0$ according as $r_s+ r_t \leq p$ or $r_s + r_t > p.$
	\end{enumerate}
\end{lemma}

In addition to the above notation, set $r=r_u, r'= r_{u-z}, q=q_u, q'= q_{u-z}$, and $W= [u-z,u-1]\times[v-w,v-1].$ Note that if $w=0$ or $z=0$, then $W$ is empty.

\begin{lemma}[{\cite[Lemma 2.4]{patil-sengupta}}]  \label{P-B} 
	With the above notation,
	\begin{enumerate}
		\item $u \geq p+1$, $q \geq 1$ and $q' \in [0,q];$
		\item  $r_z = 
		\begin{cases} 
			r-r' & \text{ if } r>r',\\
			p+r-r' & \text{ if } r \leq r'; \\
		\end{cases}$
		
		\item if $q'=q$ and $z \neq 0,$ then $r' < r$ and $r \geq 2;$
		\item if $\mu = 0,$ then $z > p$ and $q' < q;$
		\item if $\lambda =1,$ then $w \neq 0;$ 
		\item if $\lambda = 1$ and $\mu = 0,$ then $r' < r$ and $r \geq 2.$ 
	\end{enumerate}
\end{lemma}

In this article, we aim to describe the structure of $\RF$-matrices for pseudo-Frobenius numbers of a numerical semigroup $H$ defined by an almost arithmetic sequence, using the above results. We discuss the cases $W= \emptyset$ and $W \neq \emptyset$ separately.
We use the above results throughout the paper without repeated reference.

\section{The case $W = \emptyset$}
Throughout this section we will assume that $W = \emptyset$ and $H = \langle m_0,m_1,\ldots,m_p,n \rangle$; where $m_i = m_0 + id$ for $i \in [1,p]$ and $\gcd(m_0,n,d)=1.$ 

Let $\gamma_k = g_{(q-1)p+r+k-1}+(v-1)n-m_0,$ for $k \in [1,p].$ In this case, using \cite[Proposition 3.3]{patil-sengupta}, it follows that the set of pseudo-Frobenius numbers is
\[ \PF(H) = 
	\begin{cases} 
      \big\{ \gamma_k \mid k\in [1,p] \big\} & \text{ if } r=1,\\
      \big\{ \gamma_k \mid k\in [p-r+2,p] \big\} & \text{ if } r\geq 2.\\
   \end{cases}
\]
We find the factorizations of $\gamma_k + m_j$ for $j \in [0,p]$, and $\gamma_k + n$ in $H.$ From these factorizations, we form an $\RF$-matrix for $\gamma_k$.

\begin{proposition}\label{r=1,L>1}
	Let $r=1$ and $\lambda > 1.$ Then for every element $\gamma_k$ of $\PF(H)$, $k \in [1,p]$, there exists an $\RF$-matrix of the following type
	\begin{multline*}
	\RF(\gamma_k) =
	\begin{bmatrix}
	-1 & 0 & 0 & \cdots& \cdots & \cdots & 0 & q-1 & v-1\\
	0 & -1 & 0 & \cdots & \cdots & \cdots & 0 & q-1 & v-1\\
	\vdots & \vdots & \vdots & \vdots & \vdots & \vdots  & \vdots  & \vdots  & \vdots \\
	0 & 0 & \cdots & -1 & 0 & \cdots & 0 & q-1 & v-1\\
	\lambda -2 & 0 & 0 & \cdots & -1 & 0 & \cdots & 0 & w+v-1\\
	\vdots & \vdots & \vdots & \vdots & \vdots & \vdots  & \vdots  & \vdots  & \vdots \\
	\lambda-2 & 0 & 0 & \cdots & \cdots & \cdots & 0 & -1 & w+v-1\\
	a_{p+2,1} & a_{p+2,2} & a_{p+2,3} & \cdots & \cdots &\cdots & \cdots & a_{p+2,p+1} & -1
	\end{bmatrix}
	+
	\begin{bmatrix}
 	R_{(0)} \\ R_{(1)} \\ \vdots \\  R_{(p-k)} \\ R_{(p-k+1)}\\ \vdots \\  R_{(p)} \\ R_{(n)}
	\end{bmatrix},
	\end{multline*}
	where 
	\begin{enumerate}
		\item $R_{(0)} = (a^{(0)},\ldots,a^{(k)},\ldots,a^{(n)})$ such that $a^{(k)} = 1$ and all other $a^{(j)}$'s equal to zero,
		\item for all $j \in [1,p-k],$ $R_{(j)} = (a^{(0)},\ldots,a^{(k+j)},\ldots,a^{(n)})$ such that $a^{(k+j)} = 1$ and all other $a^{(j)}$'s equal to zero,
		\item for all $j \in [p-k+1,p],$ $R_{(j)} = (a^{(0)},\ldots,a^{(k+j-p-1)},\ldots,a^{(n)})$ such that $a^{(k+j-p-1)} = 1$ and all other $a^{(j)}$'s equal to zero,
		\item if $\mu \geq 1,$ then $(a_{p+2,1},a_{p+2,2},\ldots,a_{p+2,p+1},-1) = (\mu-1, 0,\ldots,0, q+q_z-1,-1)$ and $R_{(n)} = (a^{(0)},\ldots,a^{(k)},\ldots,a^{(r_z)},\ldots,a^{(n)})$ such that $a^{(k)}= a^{(r_z)} = 1$ and all other $a^{(j)}$'s equal to zero. Otherwise, when $\mu=0,$  
		$(a_{p+2,1},a_{p+2,2},\ldots,a_{p+2,p+1},-1) = (\lambda-2,0,\ldots,0,q_{z}-1,-1)$ and $R_{(n)} = (a^{(0)},\ldots,a^{(k-1)},\ldots,a^{(r_{z})},\ldots,a^{(n)})$ such that $a^{(k-1)}= a^{(r_{z})} = 1$ and all other $a^{(j)}$'s equal to zero.
	\end{enumerate}
\end{proposition}

\begin{proof}
	Since $r=1$, $\PF(H) = \{\gamma_k \mid k \in [1,p]\}$ where $\gamma_k =g_{(q-1)p+k} + (v-1)n - m_0.$
	
	(1) $\gamma_k + m_0 = g_{(q-1)p+k} +(v-1)n =(q-1)m_p + m_k + (v-1)n.$\\
	
	(2) If $j \in [1,p-k],$ then
	\begin{align*}
		\gamma_k + m_j 
		= g_{(q-1)p+k}+(v-1)n-m_0 + m_j 
		&= g_{(q-1)p+k} + g_j+(v-1)n-m_0  \\
		&= m_0 + g_{(q-1)p+k +j} +(v-1)n-m_0 \\
		&= (q-1)m_p+m_{k+j}+(v-1)n.
	\end{align*}

	(3) If $j=p-k+1,$ then 
	\begin{align*}
		\gamma_k + m_j 
		&= g_{(q-1)p+k} + (v-1)n - m_0 + m_{p-k+1}  \\
		&= (q-1)m_p + m_k + (v-1)n - m_0 + m_{p-k+1} \\
		&= (q-1)m_p + (v-1)n - m_0 + m_{p} + m_1 \\
		&= g_u + (v-1)n - m_0 
		= (\lambda - 1)m_0+(w+v-1)n.
	\end{align*}
	
	If $j > p-k+1$ and $\lambda > 1,$ then
	\begin{align*}
		\gamma_k + m_j 
		= g_{(q-1)p+k}+(v-1)n-m_0 + m_j 
		&= (q-1)m_p + m_k +(v-1)n-m_0 + m_j \\
		&= q m_p + (v-1)n - m_0 + m_{k+j-p} \\
		&= g_u + (v-1)n - 2m_0 + m_{k+j-p-1} \\
		&= (\lambda - 2)m_0 + m_{k+j-p-1} + (w+v-1)n.
	\end{align*}

	(4) If $\mu >0$, then
	\begin{align*}
		\gamma_k + n 
		= g_{(q-1)p+k}+vn-m_0 
		&= g_{(q-1)p+k} + (\mu - 1)m_0+g_z \\
		&= (q-1)m_p +m_k + (\mu - 1)m_0+q_zm_p + m_{r_z}  \\
		&= (\mu -1)m_0 +m_k + m_{r_z}+ (q + q_z -1)m_p. 
	\end{align*}
	
	If $\mu = 0,$ then $z > p$. This implies that $w=0,$ $q_{z-p} = q_z-1$ and $r_{z-p} = r_z$. Thus
	\begin{align*}
		\gamma_k + n 
		= g_{(q-1)p+k}+vn-m_0 
		&= g_{(q-1)p+k} - m_0 + g_z \\
		&= g_{(q-1)p+k} - m_0 + g_p + g_{z-p} \\
		&= g_{qp+k} - m_0 + g_{z-p} \\
		&= q m_p +m_k - m_0 + g_{z-p}  \\
		&= g_u + m_{k-1} - 2m_0 + g_{z-p} \\
		&= (\lambda-2)m_0 + m_{k-1} + (q_z -1)m_p + m_{r_{z}}.
		\qedhere
	\end{align*}
\end{proof}

\begin{proposition}\label{r=1,L=1}
	Let $r=1$ and $\lambda =1.$ Then  for every element $\gamma_k$ of $\PF(H),$ $k \in [1,p]$, there exists an $\RF$-matrix of the following type
	\begin{multline*}
		\RF(\gamma_k)=
		\begin{bmatrix}
			-1 & 0 & 0 & \cdots& \cdots & \cdots & 0 & q-1 & v-1\\
			0 & -1 & 0 & \cdots & \cdots & \cdots & 0 & q-1 & v-1\\
			\vdots & \vdots & \vdots & \vdots & \vdots & \vdots  & \vdots  & \vdots  & \vdots \\
			0 & 0 & \cdots & -1 & 0 & \cdots & 0 & q-1 & v-1\\
			\mu -1 & 0 & 0 & \cdots & -1 & 0 & \cdots & 0 & w-1\\
			\vdots & \vdots & \vdots & \vdots & \vdots & \vdots  & \vdots  & \vdots  & \vdots \\
			\mu-1 & 0 & 0 & \cdots & \cdots & \cdots & 0 & -1 & w-1\\
			\mu-1 & 0 & 0 & \cdots & \cdots &\cdots & \cdots & q-1 & -1\\
		\end{bmatrix}
		+
		\begin{bmatrix}
			R_{(0)} \\ R_{(1)} \\ \vdots \\  R_{(p-k)} \\ R_{(p-k+1)}\\ \vdots \\  R_{(p)} \\ R_{(n)}\\
		\end{bmatrix},
	\end{multline*}
	where
	\begin{enumerate}
		\item $R_{(0)} = (a^{(0)},\ldots,a^{(k)},\ldots,a^{(n)})$ such that $a^{(k)} = 1$ and all other $a^{(j)}$'s equal to zero,
		\item for all $j \in [1,p-k]$, $R_{(j)} = (a^{(0)},\ldots,a^{(k+j)},\ldots,a^{(n)})$  such that $a^{(k+j)} = 1$ and all other $a^{(j)}$'s equal to zero,
		\item for all $j \in [p-k+1,p]$, $R_{(j)} = (a^{(0)},\ldots,a^{(k+j-p-1)},\ldots,a^{(n)})$  having $a^{(k+j-p-1)} = 1$ and all other $a^{(j)}$'s equal to zero,             
		\item $R_{(n)} = (a^{(0)},\ldots,a^{(k)},\ldots,a^{(n)})$ such that $a^{(k)} = 1$ and all other $a^{(j)}$'s equal to zero.
	\end{enumerate}
\end{proposition}

\begin{proof} 
	For (1) and (2), see the proof of Proposition \ref{r=1,L>1}.\\
	
	(3) Since $\lambda = 1$ and $r=1$, using Lemma \ref{P-B} we get $w,\mu \neq 0$. Also, $W = \emptyset$ implies that $z = 0$ and hence $\nu = \mu +1.$ Put $\ell = k+j-p$. Since $k+j \geq p+1$, we get $\ell \in [1,p]$ and
	\begin{align*}
		\gamma_k + m_j 
		= g_{(q-1)p + k} + (v-1)n -m_0 + m_j
		&=g_{(u+l-1)} + (v-1)n -m_0\\
		&=g_{(u+l-1)} + \nu m_0 - g_u + (w-1)n-m_0\\
		& = (\nu -2)m_0+m_{\ell - 1} + (w-1)n\\
		&=(\mu-1)m_0+m_{k+j-p-1}+(w-1)n.
	\end{align*}
	
	(4) Since $z=0,$ we have $g_z =0$ and therefore, 
	\begin{align*}
		\gamma_k + n 
		= g_{(q-1)p+k}+vn-m_0
		&= g_{(q-1)p+k} + (\mu - 1)m_0\\
		&= (q-1)m_p +m_k + (\mu - 1)m_0.
		\qedhere
	\end{align*} 
\end{proof}

\begin{coro}\label{p=1}
	Suppose $p=1.$ Then $\PF(H) = \{\gamma_1\}$ and there exists an $\RF$-matrix of $\gamma_1$ of either of the following type
	\begin{align*}
		\begin{bmatrix}
			-1 &  u-1 & v-1     \\
			\lambda-1 &  -1 & w+v-1    \\
			\mu-1 & u+z-1 & -1 
		\end{bmatrix}
		,
		\begin{bmatrix}
			-1  &  u-1 & v-1      \\
			\lambda-1  &  -1   & w+v-1  \\
			\lambda-1 & z-1 & -1 
		\end{bmatrix}
		, \text{ or }
		\begin{bmatrix}
			-1  &  u-1 & v-1      \\
			\mu  &  -1   & w-1  \\
			\mu-1 & u-1 & -1 
		\end{bmatrix}
		.
	\end{align*}
\end{coro}

\begin{proof}
	Since $p=1,$ we get $r_t=1,~ q_t=t-1$ for all $t$ and $\PF(H)=\{\gamma_1 \}.$ Therefore, we have $k=1.$ Let $\lambda >1.$ Using Proposition \ref{r=1,L>1}, we have
	\begin{align*}
	\gamma_1+m_0 &= g_q+(v-1)n = (u-1)m_1+(v-1)n,  \\
	\gamma_1+m_1 &= (\lambda-1)m_0+(w+v-1)n.
	\end{align*}
	If $\mu >0$, then 
	\[ \gamma_1+n = (\mu-1)m_0+ (q+q_z+1)m_1 = (\mu-1)m_0+(u+z-1)m_1. \]
	If $\mu = 0,$ then $\gamma_1+n = (\lambda-1)m_0 + (q_{z-1}+1)m_1 = (\lambda-1)m_0 + (z-1)m_1.$

	Let $\lambda=1.$ Using Proposition \ref{r=1,L=1}, we get
	\begin{align*}
		\gamma_1 + m_0 &= (u-1)m_1+(v-1)n, \\  
		\gamma_1 + m_1 &= \mu m_0 + (w-1)n,  \\
		\gamma_1 + n &= (\mu - 1)m_0 + q m_1.
	\end{align*}
	Thus, we have the above possible three matrices.
\end{proof}

\begin{proposition}\label{r geq 2}
	Let $r \geq 2.$ Then  for every element $\gamma_k$ of $\PF(H)$, $k\in [p-r+2,p],$ there exists an $\RF$-matrix of the following type
	\begin{multline*}
		\RF(\gamma_k)
		=
		\begin{bmatrix}
			-1 & 0 & 0 & \cdots& \cdots & \cdots & 0 & q & v-1\\
			0 & -1 & 0 & \cdots & \cdots & \cdots & 0 & q & v-1\\
			\vdots & \vdots & \vdots & \vdots & \vdots & \vdots  & \vdots  & \vdots  & \vdots \\
			0 & 0 & \cdots & -1 & 0 & \cdots & 0 & q & v-1\\
			\lambda -1 & 0 & 0 & \cdots & -1 & 0 & \cdots & 0 & w+v-1\\
			\vdots & \vdots & \vdots & \vdots & \vdots & \vdots  & \vdots  & \vdots  & \vdots \\
			\lambda-1 & 0 & 0 & \cdots & \cdots & \cdots & 0 & -1 & w+v-1\\
			a_{p+2,1} & a_{p+2,2} & a_{p+2,3} & \cdots & \cdots &\cdots & \cdots & a_{p+2,p+1} & -1\\
		\end{bmatrix}
		+
		\begin{bmatrix}
			R_{(0)} \\ R_{(1)} \\ \vdots \\  R_{(e)} \\ R_{(e+1)}\\ \vdots \\  R_{(p)} \\ R_{(n)}\\
		\end{bmatrix},
	\end{multline*}
	where $e \in [1,p-1]$ such that $r+k+e = 2p+1$ and
	\begin{enumerate}
		\item $R_{(0)} = (a^{(0)},\ldots,a^{(p-e)},\ldots,a^{(n)})$ such that $a^{(p-e)} = 1$ and all other $a^{(j)}$'s equal to zero,
		
		\item for all $j \in [1,e]$, $R_{(j)} = (a^{(0)},\ldots,a^{(p-e+j)},\ldots,a^{(n)})$ such that $a^{(p-e+j)} = 1$ and all other $a^{(j)}$'s equal to zero,
		
		\item for all $j \in [e+1,p]$, $R_{(j)} = (a^{(0)},\ldots,a^{(k+j-p-1)},\ldots,a^{(n)})$  such that $a^{(k+j-p-1)} = 1$ and all other $a^{(j)}$'s equal to zero,
		\item  if $\mu \geq 1,$ then $(a_{p+2,1},a_{p+2,2},\ldots,a_{p+2,p+1},-1)= (\mu-1, 0,\ldots,0, q+q_z,-1)$ and $R_{(n)} = (a^{(0)},\ldots,a^{(p-e)},\ldots,a^{(r_z)},\ldots,a^{(n)})$ such that $a^{(p-e)}= a^{(r_z)} = 1$ and all other $a^{(j)}$'s equal to zero. Otherwise, when $\mu = 0$ then $ (a_{p+2,1},a_{p+2,2},\ldots,a_{p+2,p+1},-1)=(\lambda-1,0,\ldots,0,q_{z}-1,-1)$ and $R_{(n)} = (a^{(0)},\ldots,a^{(k-1)},\ldots,a^{(r_{z})},\ldots,a^{(n)})$ such that $a^{(k-1)}= a^{(r_{z})} = 1$ and all other $a^{(j)}$'s equal to zero. 
	\end{enumerate}
\end{proposition}

\begin{proof}
	Since $r \geq 2$, $\PF(H) = \{\gamma_k \mid k \in[p-r+2,p]\}.$ Therefore,\\
	
	(1) $\gamma_k + m_0= g_{(q-1)p+r+k-1}+(v-1)n =qm_p+ m_{p-e}+(v-1)n.$ \\
	
	(2) If $j \in [1,e]$, then $r+k+j-p-1 \leq p$ and
	\begin{align*}
		\gamma_k + m_j 
		= g_{(q-1)p+r+k-1}+(v-1)n-m_0 + m_j 
		&=m_0 + qm_p+m_{p-e+j} +(v-1)n-m_0\\ 
		&= qm_p + m_{p-e+j} + (v-1)n.
	\end{align*}
	
	(3) If $j \in [e+1,p],$ then $r+k+j-p-1 > p$ and
	\begin{align*}
		\gamma_k + m_j 
		= g_{(q-1)p+r+k-1}+(v-1)n-m_0 + m_j
		&=g_{u-p+j+k-1}-(v-1)n-m_0\\
		&=g_u -m_0 + g_{j+k-p-1}-(v-1)n\\
		&= (\lambda - 1)m_0 + wn + m_{k+j-p-1}+(v-1)n \\
		&= (\lambda - 1)m_0 + m_{k+j-p-1} + (w+v-1)n.
	\end{align*} 

	(4) If $\mu \geq 1,$ then 
	\begin{align*}
		\gamma_k + n
		= g_{(q-1)p+r+k-1}+(v-1)n-m_0+ n 
		&=qm_p + m_{p-e} + (\mu -1)m_0 + g_z\\
		&=qm_p + m_{p-e} + (\mu -1)m_0 + q_z m_p + m_{r_z}\\
		&=(\mu -1)m_0  + m_{p-e} +  m_{r_z}+(q + q_z)m_p.
	\end{align*}
	
	If $\mu = 0,$ then $z > p.$ This implies that $w=0,$ $q_{z-p} = q_z-1$ and $r_{z-p} = r_z$. Thus 
	\begin{align*}
		\gamma_k + n
		= g_{(q-1)p+r+k-1}+(v-1)n-m_0+ n 
		&=g_{u-p+k-1}+vn-m_0\\
		&=g_u+g_{z-p}+g_{k-1}-m_0\\
		&= (\lambda - 1)m_0 + m_{k-1}+ g_{z-p}\\
		&= (\lambda - 1)m_0 + m_{k-1}+ (q_{z}-1) m_p + m_{r_{z}}.
		\qedhere
	\end{align*}
\end{proof}

\begin{example} 
{\rm Let $H = \langle m_0 = 14, m_1 = 17, m_2 =20, m_3 = 23, m_4 = 26, n= 21 \rangle.$ Then $H' = \langle 14, 17, 20, 23, 26 \rangle$ and $p=4.$ The Ap\'{e}ry set of $H$ with respect to $14$ is
	\begin{align*}
		\Ap(H,14) = \{0,17,20,21,23,26,38,41,43,44,46,47,64,67 \}.
	\end{align*} 
	The smallest $i$ such that $g_i \notin \Ap(H,14)$ is $g_7$ $(= 49).$ Hence, $u=7,$ $q=1,$ and $r=3$. Since $21 \notin H'$ but $42 \in H',$ we have $v=2.$ Also, $49 = 21+ 2 \times 14$, $42 = 3 \times 14,$ implies that $w=1$, $z=0,$ $\lambda=2,$ $\mu=3,$ $r'=3,$  $q'=1,$ and $\nu=5.$ Since $r \geq 2$, we get $\PF(H) = \{\gamma_3,\gamma_4\} = \{50,53\}.$ Using Proposition {\rm \ref{r geq 2}},
	\begin{align*}
	\begin{array}{ c c c }
  	\gamma_3 + m_0 = m_1 + m_4 + n, & \gamma_3 + m_1 = m_2 + m_4 + n, & \gamma_3 + m_2 = m_3 + m_4 + n, \\ 
 	\gamma_3 + m_3 = 2m_4 + n, & \gamma_3 + m_4 = m_0 + m_2 + 2n, & \gamma_3 + n= 2m_0 + m_1 + m_4. 
	\end{array}
	\end{align*}
	We can therefore write an $\RF$-matrix for $\gamma_3 = 50$ from these expressions,
	\begin{align*}
	\RF(\gamma_3) = \left[	
	\begin{array}{c c c c c c}
	-1 & 1 & 0 & 0 & 1 & 1\\
	0 & -1 & 1 & 0 & 1 & 1\\
	0 & 0 & -1 & 1 & 1 & 1\\
	0 & 0 & 0 & -1 & 2 & 1\\
	1 & 0 & 1 & 0 & -1 & 2\\
	2 & 1 & 0 & 0 & 1 & -1
	\end{array}
	\right].
	\end{align*}
	Similarly,
	\begin{align*}
		\begin{array}{ c c c }
			\gamma_4 + m_0 = m_2 + m_4 + n, & \gamma_4 + m_1 = m_3 + m_4 + n, & \gamma_4 + m_2 =2m_4 + n,\\
			\gamma_4 + m_3 = m_0 + m_2 + 2n, & \gamma_4 + m_4 = m_0 + m_3 + 2n, & \gamma_4 + n= 2m_0 + m_2 + m_4.\\ 
		\end{array}
	\end{align*}
	The $\RF$-matrix for $\gamma_4 = 53$ from the above expressions is as follows\\
	\[ \RF(\gamma_4) =
	\left[
	\begin{array}{c c c c c c}
		-1 & 0 & 1 & 0 & 1 & 1\\
		0 & -1 & 0 & 1 & 1 & 1\\
		0 & 0 & -1 & 0 & 2 & 1\\
		1 & 0 & 1 & -1 & 0 & 2\\
		1 & 0 & 0 & 1 & -1 & 2\\
		2 & 0 & 1 & 0 & 1 & -1
	\end{array}
	\right].
	\]
	Using the software GAP \cite{GAP4},  package \texttt{numericalsgps} \cite{gap-ns}, one can check that the number of $\RF$-matrices for 50 are 720 and that for 53 are 2520.
	}
\end{example}

\begin{coro} \label{r=2,ne}
	Let $r=2$. Then H is symmetric and there exists an $\RF(H)$ of the following type
	\[ \RF(H) =
	\left[
	\begin{array}{c c c c c c c c c }
		-1 & 1 & 0 & 0 & \cdots & \cdots & 0 & q & v-1\\
		0 & -1 & 1 & 0  & \cdots & \cdots & 0 & q & v-1\\
		\vdots & \vdots & \vdots & \vdots  &\vdots & \vdots & \vdots & \vdots & \vdots\\
		\vdots & \vdots & \vdots & \vdots  &\vdots & \vdots & \vdots & \vdots & \vdots\\
		0 & 0 & 0 & 0 & \cdots & -1 & 1 & q &v-1\\
		0 & 0 & 0 & 0  & \cdots & \cdots & -1 & q+1 & v-1 \\
		\lambda -1 & 0 & 0 & 0 & \cdots & \cdots & 1 & -1 & w+v-1\\
		a_{p+2,1} & a_{p+2,2} & a_{p+2,3} & \cdots & \cdots &\cdots & \cdots & a_{p+2,p+1} & -1
	\end{array}
	\right],
	\]
	where $(a_{p+2,1},\ldots,a_{p+2,p+1},-1) = (\mu -1,1,0,\ldots,0,q+q_z,-1)+ (a^{(0)},\ldots,a^{(r_z)},\ldots,a^{(n)})$ such that $a^{(r_z)} = 1$ and all other $a^{(j)}$'s equal to zero when $\mu > 0$. Otherwise, when $\mu = 0$, we have
	\begin{center}
		$(a_{p+2,1},\ldots,a_{p+2,p+1},-1) = (\lambda -1,0,\ldots,0,1,q_{z}-1,-1)+(a^{(0)},\ldots,a^{(r_{z})},\ldots,a^{(n)})$ 
	\end{center} 
	such that $a^{(r_{z})} = 1$ and all other $a^{(j)}$'s equal to zero.
\end{coro}

\begin{proof}
	Since $r=2,$ we get $k=p$. Therefore, $\PF(H)=\{\gamma_p\}$ and so $H$ is symmetric. Also, as $k=p$ and $r=2,$ we have $r+k+j-p-1 > p$ only for $j=p.$ Therefore,
	\begin{align*}
		\gamma_p + m_0  &= g_{qp+1}+(v-1)n = qm_p + m_1 +(v-1)n,  \\
		\gamma_p + m_j &=  m_{j+1}+ qm_p +(v-1)n, \text{ for all } j \in [1,p-1],  \\
		\gamma_p + m_p &= (\lambda - 1)m_0 + m_{p-1} + (w+v-1)n.
	\end{align*}
	If $\mu > 1,$ then $\gamma_p + n = (\mu-1)m_0 + m_1 + m_{r_z} + (q + q_z)m_p,$ and
	if $\mu = 0,$ then $\gamma_p + n = (\lambda-1)m_0 + m_{p-1} + m_{r_{z}} + ( q_{z}-1)m_p$ . Hence the result follows.
\end{proof}

\begin{example} 
{\rm Let $H = \langle m_0 = 10, m_1 = 19, m_2 =28, m_3 = 37, n= 35 \rangle.$ Then,
	$p = 3,$ $u=5,$ $v=2,$ $w=1,$ $z=0,$ $\lambda=3,$ $\mu=7,$ $\nu=10,$ $r=2,$ $q=1.$ Since $z=0$ and $r = 2 $, therefore $W=\emptyset$ and $\PF(H) = \{\gamma_3\} = \{81\}.$\\
	\[ \RF(H) =
	\left[
	\begin{array}{c c c c c }
		-1 & 1 & 0 & 1 & 1\\
		0 & -1 & 1 & 1 & 1\\
		0 & 0 & -1 & 2 & 1\\
		2 & 0 & 1 & -1 & 2\\
		6 & 1 & 0 & 1 & -1
	\end{array}
	\right].
	\]
	}
\end{example}

\section{The case $W \neq \emptyset$}

Throughout this section we assume $W \neq \emptyset$ and $H = \langle m_0,m_1,\ldots,m_p,n \rangle$, where $m_i = m_0 + id$ for $i \in [1,p]$ and $\gcd(m_0,n,d)=1.$ In this case, there are two types of pseudo-Frobenius numbers. We denote them by $\PF_1(H)$ and $\PF_2(H).$\\

$\bullet$ Let $\alpha_i = g_{(q'-1)p+i}+(v-1)n-m_0.$ Then $\PF_1(H) \subset \{\alpha_i ; ~i \in I\}$, where\[ I= \begin{cases} 
      \{p\} & \text{ if } r'=1 \text{ and } q'=0,\\
      [1,p] & \text{ if } r'=1 \text{ and } q'>0,\\
      [p+1,p+r'-1]& \text{ if } r'\geq 2. 
   \end{cases}
\]

$\bullet$  Let $\beta_j = g_{(q-1)p+r+j-1}+(v-w-1)n-m_0$ for $j \in [1,p].$ Then,
 \[ \PF_2(H) \subset \begin{cases} 
      \big\{ \beta_j \mid j \in [1,p] \big\} & \text{ if } r=1,\\
      \big\{ \beta_j \mid j \in [p-r+2,p] \big\} & \text{ if } r\geq 2.
   \end{cases}
\]
For the complete description of $\PF_1(H)$ and $\PF_2(H)$, see {\cite[Propositions 4.6 and 4.7]{patil-sengupta}} and {\cite[Propositions 5.6, 5.7, 5.8, 5.9]{patil-sengupta}}.

\subsection{Structure of $\RF$-matrices for pseudo-Frobenius numbers in $\PF_1(H)$}

\begin{proposition}\label{r'=1,mu>0,q'=0}
	Let $r'=1, \mu >0$ and $q'=0.$ Then $\PF_1(H) = \{\alpha_p\}$ and there exists an $\RF(\alpha_p)$ of the following type
	\[ \RF(\alpha_p) =
	\left[
	\begin{array}{c c c c c c c c c }
		-1 & 0 & 0 & 0 & \cdots & \cdots & 0 & 0 & v-1\\
		\nu -1 & -1 & 0 & 0  & \cdots & \cdots & 0 & 0 & w-1\\
		\nu -2 & 1 & -1 & 0 & \cdots & \cdots & 0 & 0 & w-1\\
		\nu-2 & 0 & 1 & -1 & \cdots & \cdots & 0 & 0 &w-1\\
		\vdots & \vdots & \vdots & \vdots  &\vdots & \vdots & \vdots & \vdots & \vdots\\
		\vdots & \vdots & \vdots & \vdots  &\vdots & \vdots & \vdots & \vdots & \vdots\\
		\nu -2 & 0 & 0 & 0 & \cdots & \cdots & 1 & -1 & w-1\\
		a_{p+2,1} & a_{p+2,2} & a_{p+2,3} & \cdots & \cdots &\cdots & \cdots & a_{p+2,p+1} & -1
	\end{array}
	\right],
	\]
	where $(a_{p+2,1},\ldots,a_{p+2,p+1},-1) = (\mu -1,0,\cdots,0,q_z,-1)+ (a^{(0)},\ldots,a^{(r_z)},\ldots,a^{(n)})$ such that $a^{(r_z)} = 1$ and all other $a^{(i)}$'s equal to zero.
\end{proposition}

\begin{proof}
	Since $r'=1$ and $q'=0$, we get $\PF_1(H)=\{\alpha_p\}.$ Also, $u-z = q'p+r'=1.$\\
	(i) $\alpha_p + m_0 = g_0 + (v-1)n = (v-1)n.$ \\
	(ii) For $j\in[1,p]$, 
	\begin{align*}
		\alpha_p + m_j
		= (v-1)n-m_0+m_j
		&=(v-w)n+(w-1)n-m_0+m_j\\
		&=\nu m_0 - g_{(u-z)}+(w-1)n-m_0+m_j \\
		&=(\nu -2)m_0+m_{j-1}+(w-1)n.
	\end{align*}
	(iii) $\alpha_p + n = (v-1)n-m_0+n = (\mu-1)m_0 + g_z = (\mu -1)m_0 + q_zm_p + m_{r_z}.$\\
	Hence, the result follows.
\end{proof}

\begin{example}
{\rm Let $H = \langle m_0 = 11, m_1 = 13, m_2 =15, m_3 = 17, m_4 = 19, n= 21 \rangle.$ Then,\\
	$p=4,u=5,v=3,w=1,z=4,\lambda=1, \mu=4, \nu=5, r=1, r'=1, q=1, q'=0. $ Since $r' = 1, q'=0$, we get $\PF_1(H) = \{\alpha_4\} =\{31\}$ and
	\[ \RF(31) =
	\left[
	\begin{array}{c c c c c c}
		-1 & 0 & 0 & 0 & 0 & 2\\
		4 & -1 & 0 & 0 & 0 & 0\\
		3 & 1 & -1 & 0 & 0 & 0\\
		3 & 0 & 1 & -1 & 0 & 0\\
		3 & 0 & 0 & 1 & -1 & 0\\
		3 & 0 & 0 & 0 & 1 & -1
	\end{array}
	\right].
	\]
	}
\end{example}

\begin{proposition}\label{r'=1,mu >0,q'>0}
	Let $r'=1,\mu >0$ and $q'>0.$ Then for each $\alpha_i \in \PF_1(H)$, $ i \in [1,p]$, there exists an $\RF(\alpha_i)$ of the following type
	\begin{multline*}
		\RF(\alpha_i)
		=
		\begin{bmatrix}
			-1 & 0 & 0 & \cdots& \cdots & \cdots & 0 & q'-1 & v-1\\
			0 & -1 & 0 & \cdots & \cdots & \cdots & 0 & q'-1 & v-1\\
			\vdots & \vdots & \vdots & \vdots & \vdots & \vdots  & \vdots  & \vdots  & \vdots \\
			0 & 0 & \cdots & -1 & 0 & \cdots & 0 & q'-1 & v-1\\
			\nu -2 & 0 & 0 & \cdots & -1 & 0 & \cdots & 0 & w-1\\
			\vdots & \vdots & \vdots & \vdots & \vdots & \vdots  & \vdots  & \vdots  & \vdots \\
			\nu-2 & 0 & 0 & \cdots & \cdots & \cdots & 0 & -1 & w-1\\
			\mu-1 & 0 & 0 & \cdots & \cdots &\cdots & \cdots & q'+q_z-1 & -1\\
		\end{bmatrix}
		+
		\begin{bmatrix}
			R_{(0)} \\ R_{(1)} \\ \vdots \\  R_{(p-i)} \\ R_{(p-i+1)}\\ \vdots \\  R_{(p)} \\ R_{(n)}\\
		\end{bmatrix},
	\end{multline*}
	where 
	\begin{enumerate}
		\item $R_{(0)} = (a^{(0)},\ldots,a^{(i)},\ldots,a^{(n)})$ such that $a^{(i)} = 1$ and all other $a^{(j)}$'s equal to zero,
		
		\item for all $j \in [1,p-i]$, $R_{(j)} = (a^{(0)},\ldots,a^{(i+j)},\ldots,a^{(n)})$  such that $a^{(i+j)} = 1$ and all other $a^{(j)}$'s equal to zero,
		
		\item for all $j \in [p-i+1,p]$, $R_{(j)} = (a^{(0)},\ldots,a^{(i+j-p-1)},\ldots,a^{(n)})$  such that $a^{(i+j-p-1)} = 1$ and all other $a^{(j)}$'s equal to zero,
		
		\item $R_{(n)} = (a^{(0)},\ldots,a^{(i)},\ldots,a^{(r_z)},\ldots,a^{(n)})$ such that $a^{(r_z)},a^{(i)} = 1$ and all other $a^{(j)}$'s equal to zero.
	\end{enumerate}
\end{proposition}

\begin{proof} 
	Since $\mu>0$ and $q'>0,$ we have $\PF_1(H) = \{\alpha_i \mid i\in [1,p]\}$.\\
	
	(1) $\alpha_i +m_0= g_{(q'-1)p+i} + (v-1)n=(q'-1)m_p + m_i + (v-1)n.$\\
	
	(2) If $j \in [1,p-i],$ then $i+j \leq p$ and 
	\begin{align*}
		\alpha_i +m_j
		= g_{(q'-1)p+i} + (v-1)n -m_0 +m_j
		&=(q'-1)m_p+m_i + (v-1)n -m_0 +m_j\\
		&=(q'-1)m_p+m_{i+j} + (v-1)n.
	\end{align*}
	
	(3) If $j \in [p-i+1,p],$ then $i+j > p$ and
	\begin{align*}
		\alpha_i +m_j
		= g_{(q'-1)p+i} + (v-1)n -m_0 +m_j  &=g_{(q'-1)p+i+j} + (v-1)n -m_0 \\
		&= g_{u-z-r'-p+i+j} + (v-w)n + (w-1)n-m_0\\ 
		&= g_{i+j-p-1}+(\nu-2)m_0+(w-1)n\\
		&=(\nu-2)m_0+m_{i+j-p-1}+(w-1)n.
	\end{align*}
	
	(4) Now,
	\begin{align*}
		\alpha_i +n
		= g_{(q'-1)p+i} + vn -m_0
		&= g_{(q'-1)p+i} + \mu m_0 +g_z -m_0\\ 
		&= (q'+q_z-1)m_p + (\mu -1)m_0 + m_i + m_{r_z}.
		\qedhere
	\end{align*}
\end{proof}

Note that if $p=1$ and $\PF_1(H) \neq \emptyset,$ then $\PF_1(H) = \{\alpha_1\}.$

\begin{coro}\label{p=1,PF_1(H)}
	Suppose $p=1$ and $\PF_1(H) \neq \emptyset.$ Then there exists an $\RF$-matrix of $\alpha_1$ of the type
	\[
	\begin{bmatrix}
		-1 &  u-z-1 & v-1     \\
		\nu-1 &  -1 & w-1    \\
		\mu-1 & u-1 & -1 
	\end{bmatrix}
	,\quad \text{ or }
	\begin{bmatrix}
		-1  &  0 & v-1      \\
		\nu-1  &  -1   & w-1  \\
		\mu-1 & z & -1 
	\end{bmatrix}.
	\]
\end{coro}

\begin{proof}
	Since $p=1,$ we get $r_t = 1, q_t = t-1$ for all $t.$ Now use Propositions \ref{r'=1,mu>0,q'=0} and \ref{r'=1,mu >0,q'>0} to get desired $\RF$-matrices accordingly as $q'>0$ or $q'=0$. 
\end{proof}

\begin{proposition}\label{r'=1,mu=0,q'>0,r geq 2}
	Let $r'=1,\mu =0,q'>0$ and $r\geq 2$. Then for every element of $\PF_1(H) = \{\alpha_i \mid i\in [1,p-r+1]\}$, there exists an $\RF(\alpha_i)$ of the following type
	\begin{multline*}
		\RF(\alpha_i)
		=
		\begin{bmatrix}
			-1 & 0 & 0 & \cdots& \cdots & \cdots & 0 & q'-1 & v-1\\
			0 & -1 & 0 & \cdots & \cdots & \cdots & 0 & q'-1 & v-1\\
			\vdots & \vdots & \vdots & \vdots & \vdots & \vdots  & \vdots  & \vdots  & \vdots \\
			0 & 0 & \cdots & -1 & 0 & \cdots & 0 & q'-1 & v-1\\
			\nu -2 & 0 & 0 & \cdots & -1 & 0 & \cdots & 0 & w-1\\
			\vdots & \vdots & \vdots & \vdots & \vdots & \vdots  & \vdots  & \vdots  & \vdots \\
			\nu-2 & 0 & 0 & \cdots & \cdots & \cdots & 0 & -1 & w-1\\
			0 & 0 & 0 & \cdots & \cdots &\cdots & 0 & q-1 & -1\\
		\end{bmatrix}
		+
		\begin{bmatrix}
			R_{(0)} \\ R_{(1)} \\ \vdots \\  R_{(p-i)} \\ R_{(p-i+1)}\\ \vdots \\  R_{(p)} \\ R_{(n)}\\
		\end{bmatrix},
	\end{multline*}
	where
	\begin{enumerate}
		\item $R_{(0)} = (a^{(0)},\ldots,a^{(i)},\ldots,a^{(n)})$ such that $a^{(i)} = 1$ and all other $a^{(j)}$'s equal to zero,
		
		\item for all $j \in [1,p-i]$, $R_{(j)} = (a^{(0)},\ldots,a^{(i+j)},\ldots,a^{(n)})$ such that $a^{(i+j)} = 1$ and all other $a^{(j)}$'s equal to zero,
		
		\item For all $j \in [p-i+1,p]$, $R_{(j)} = (a^{(0)},\ldots,a^{(i+j-p-1)},\ldots,a^{(n)})$ such that $a^{(i+j-p-1)} = 1$ and all other $a^{(j)}$'s equal to zero.
		
		\item $R_{(n)} = (a^{(0)},\ldots,a^{(r+i-1)},\ldots,a^{(n)})$ such that $a^{(r+i-1)} = 1$ and all other $a^{(j)}$'s equal to zero.
	\end{enumerate}
\end{proposition}

\begin{proof}
	For (1),(2) and (3), see the proof of Proposition \ref{r'=1,mu >0,q'>0}. \\
	
	(4) Since $r \geq 2$ and $r'=1,$ we get $r_z=r-1.$ Also, we have $i \in [1,p-r+1]$ and hence $i+ r_z \leq p.$ Therefore,
	\begin{align*}
		\alpha_i +n
		= g_{(q'-1)p+i} + vn -m_0
		&= g_{(q'-1)p+i} +g_z -m_0\\
		&=g_{u-z-p+i-1} +g_z -m_0 \\
		&= g_{u-p+i-1} \\
		&= g_{(q-1)p + r+i-1} 
		=(q-1)m_p+m_{r+i-1}.
		\qedhere
	\end{align*}
\end{proof}

Note that $\PF(H) = \emptyset$ in the following cases and so we proceed to the case $r' \geq 2.$
\begin{align*}
	{\rm(1)} \; r'=1, \mu=0 \text{ and } q'=0,  \qquad {\rm(2)} \; r'=1, \mu=0, q'>0, \text{ and } r=1.
\end{align*}

\begin{proposition}\label{r' geq 2}
	Let $r'\geq 2$. Then for each $\alpha_i \in \PF_1(H)$, $i \in [p+1, p+r'-1]$, there exists an $\RF(\alpha_i)$ of the following type
	\begin{multline*}
		\RF(\alpha_i)
		=
		\begin{bmatrix}
			-1 & 0 & 0 & \cdots& \cdots & \cdots & 0 & q' & v-1\\
			0 & -1 & 0 & \cdots & \cdots & \cdots & 0 & q' & v-1\\
			\vdots & \vdots & \vdots & \vdots & \vdots & \vdots  & \vdots  & \vdots  & \vdots \\
			0 & 0 & \cdots & -1 & 0 & \cdots & 0 & q' & v-1\\
			\nu -1 & 0 & 0 & \cdots & -1 & 0 & \cdots & 0 & w-1\\
			\vdots & \vdots & \vdots & \vdots & \vdots & \vdots  & \vdots  & \vdots  & \vdots \\
			\nu-1 & 0 & 0 & \cdots & \cdots & \cdots & 0 & -1 & w-1\\
			a_{p+2,1} & a_{p+2,2} & 0 & \cdots & \cdots &\cdots & \cdots & a_{p+2,p+1} & -1 \\
		\end{bmatrix}
		+
		\begin{bmatrix}
			R_{(0)} \\ R_{(1)} \\ \vdots \\  R_{(2p-i)} \\ R_{(2p-i+1)}\\ \vdots \\  R_{(p)} \\ R_{(n)}\\
		\end{bmatrix},
	\end{multline*}
	where $\theta_i \in [1,r'-1]$ such that $i=p+\theta_i,$
	\begin{enumerate}
		\item $R_{(0)} = (a^{(0)},\ldots,a^{(\theta_i)},\ldots,a^{(n)})$ such that $a^{(\theta_i)} = 1$ and all other $a^{(j)}$'s equal to zero,
		
		\item for all $j \in [1,2p-i]$, $R_{(j)} = (a^{(0)},\ldots,a^{(\theta_i+j)},\ldots,a^{(n)})$ such that $a^{(\theta_i+j)} = 1$ and all other $a^{(j)}$'s equal to zero,
		
		\item for all $j \in [2p-i+1,p]$, $R_{(j)} = (a^{(0)},\ldots,a^{(\theta_i+j-r')},\ldots,a^{(n)})$ such that $a^{(\theta_i+j-r')} = 1$ and all other $a^{(j)}$'s equal to zero,
		
		\item if $\mu > 0$, then $(a_{p+2,1},a_{p+2,2},\ldots,a_{p+2,p+1},-1)= (\mu-1, 0,\ldots,0, q'+q_z,-1)$ and $R_{(n)} = (a^{(0)},\ldots,a^{(\theta_i)},\ldots,a^{(r_z)},\ldots,a^{(n)})$ such that $a^{(r_z)},a^{(\theta_i)} = 1$ and all other $a^{(j)}$'s equal to zero. Otherwise, if $\mu = 0$ then  $(a_{p+2,1},a_{p+2,2},\ldots,a_{p+2,p+1},-1)=(0,\ldots,0,q'+q_z,-1)$ and $R_{(n)} = (a^{(0)},\ldots,a^{(\theta_i + r_z)},\ldots,a^{(n)})$ such that $a^{(\theta_i + r_z)}= 1$ and all other $a^{(j)}$'s equal to zero.
	\end{enumerate}
\end{proposition}

\begin{proof}
	Since $r' \geq 2$, we have $\PF_1(H) \subseteq \{\alpha_i \mid i \in [p+1,p+r'-1] \}.$\\
	
	(1) $\alpha_i +m_0= g_{(q'-1)p+i} + (v-1)n= g_{q'p+\theta_i} + (v-1)n=q'm_p + m_{\theta_i} + (v-1)n.$\\
	
	(2) If $j \in [1,2p-i]$, then $i+j \leq 2p$ and hence $\theta_i + j \leq p.$ Therefore,
	\begin{align*}
		\alpha_i +m_j
		= g_{(q'-1)p+i} + (v-1)n -m_0 +m_j
		&=g_{q'p+ \theta_i} + (v-1)n - m_0+g_j\\
		&=g_{q'p+\theta_i+j} + (v-1)n \\
		&= q'm_p + m_{\theta_i + j} + (v-1)n.
	\end{align*}
	
	(3) If $j \in [2p-i+1,p]$, then $i+j > 2p$. Therefore,
	\begin{align*}
		\alpha_i +m_j 
		= g_{q'p-p+i} + (v-1)n -m_0 +m_j 
		&= g_{u-z-r'-p+i} + g_j + (v-w)n+(w-1)n-m_0 \\
		&= g_{u-z} + g_{i+j-p-r'}+(v-w)n+(w-1)n-m_0 \\
		&= (\nu-1)m_0 + m_{i+j-p-r'}+(w-1)n.
	\end{align*}

	(4) Assume $\mu >0$. Then 
	\begin{align*}
		\alpha_i + n 
		= g_{(q'-1)p+i} + vn - m_0 
		&= g_{(q'-1)p+i} + (\mu-1)m_0 + g_z \\
		&= g_{q'p + \theta_i} + (\mu-1)m_0 + g_z\\ 
		&= (\mu-1)m_0 + m_{r_z} + m_{\theta_i} + (q_z + q')m_p.
	\end{align*}

	If $\mu = 0,$ then $vn = g_z$. Therefore,
	\begin{align*}
		\alpha_i + n 
		= g_{(q'-1)p+i} + vn - m_0 
		&= g_{q'p+\theta_i} + g_z - m_0 \\
		&= q'm_p + q_zm_p  + m_{\theta_i} + m_{r_z}-m_0.
	\end{align*}
	
	Since we have $r_z + \theta_i \leq p,$ this implies $\alpha_i + n = (q' + q_z)m_p + m_{\theta_i + r_z}.$
\end{proof}

\subsection{Structure of the $\RF$-matrices for pseudo-Frobenius numbers in $\PF_2(H)$}
\begin{proposition}\label{r=1,L>1,PF_2(H)}
	Let $r=1$ and $\lambda >1$. Then for each  $\beta_j \in \PF_2(H)$, $j \in [1,p]$, there exists an $\RF(\beta_j)$ of the following type
	\begin{multline*}
		\RF(\beta_j)
		=
		\begin{bmatrix}
			-1 & 0 & 0 & \cdots& \cdots & \cdots & 0 & q-1 & v-w-1\\
			0 & -1 & 0 & \cdots & \cdots & \cdots & 0 & q-1 & v-w-1\\
			\vdots & \vdots & \vdots & \vdots & \vdots & \vdots  & \vdots  & \vdots  & \vdots \\
			0 & 0 & \cdots & -1 & 0 & \cdots & 0 & q-1 & v-w-1\\
			\lambda -2 & 0 & 0 & \cdots & -1 & 0 & \cdots & 0 & v-1\\
			\vdots & \vdots & \vdots & \vdots & \vdots & \vdots  & \vdots  & \vdots  & \vdots \\
			\lambda-2 & 0 & 0 & \cdots & \cdots & \cdots & 0 & -1 & v-1\\
			\nu-\epsilon & 0 & \cdots & \cdots & \cdots &\cdots & 0 & q_{z+j-p-1} & -1\\
		\end{bmatrix}
		+
		\begin{bmatrix}
			R_{(0)} \\ R_{(1)} \\ \vdots \\  R_{(p-j)} \\ R_{(p-j+1)}\\ \vdots \\  R_{(p)} \\ R_{(n)}\\
		\end{bmatrix},
	\end{multline*}
	where $\epsilon =2$ or $1$ according as $r'+ r_{z-p+j-1} \leq p$ or  $r'+ r_{z-p+j-1} > p$ and
	\begin{enumerate}
		\item $R_{(0)} = (a^{(0)},\ldots,a^{(j)},\ldots,a^{(n)})$ such that $a^{(j)} = 1$ and all other $a^{(k)}$'s equal to zero,
		
		\item for all $k \in [1,p-j]$, $R_{(k)} = (a^{(0)},\ldots,a^{(j+k)},\ldots,a^{(n)})$ such that $a^{(j+k)} = 1$ and all other $a^{(k)}$'s equal to zero,
		
		\item for all $k \in [p-j+1,p]$, $R_{(k)} = (a^{(0)},\ldots,a^{(k+j-p-1)},\ldots,a^{(n)})$ such that $a^{(k+j-p-1)} = 1$ and all other $a^{(k)}$'s equal to zero,
		
		\item $R_{(n)} = (a^{(0)},\ldots,a^{(r_{z+j-p-1})},\ldots,a^{(n)})$ such that $a^{(r_{z+j-p-1})} = 1$ and all other $a^{(k)}$'s equal to zero.
	\end{enumerate} 
\end{proposition}
            
\begin{proof}
	Since $r=1,$ $\beta_j = g_{(q-1)p+j} + (v-w-1)n-m_0.$\\
	
	(1) $\beta_j + m_0 = (q-1)m_p + m_j + (v-w-1)n.$ \\
	
	(2) If $k \in [1,p-j],$ then $k+j \leq p.$ Therefore,
	\begin{align*}
		\beta_j + m_k 
		= g_{(q-1)p+j} + (v-w-1)n-m_0 + m_k 
		&= (q-1)m_p + m_{j+k} + (v-w-1)n.
	\end{align*}
	
	(3) If $k \in [p-j+1,p],$ then $k+j \geq p+1.$ Given $\lambda >1,$ if $k+j > p+1$ then
	\begin{align*}
		\beta_j + m_k
		= g_{(q-1)p+j} + (v-w-1)n - m_0 + m_k 
		&= g_{u+k+j-p-1} + (v-1)n + \lambda m_0 - g_u - m_0 \\
		&= (\lambda-2)m_0 + g_{j+k-p-1} + (v-1)n \\
		&= (\lambda-2)m_0 + m_{j+k-p-1} + (v-1)n.
	\end{align*}

	If $k+j = p+1,$ then $\beta_j + m_k = (\lambda -1)m_0 + (v-1)n.$\\
	
	(4) We consider two subcases $z\geq p$ and $z < p.$
	
	(a) Assume $z \geq p.$
	\begin{align*}
		\beta_j+n 
		= g_{(q-1)p+j} + (v-w-1)n-m_0 + n 
		&= g_{(q-1)p+j} + (\nu -1)m_0 - g_{u-z}\\
		&= g_{u-z+z-p+j-1} + (\nu -1)m_0 - g_{u-z}.
	\end{align*}

	Since $z \geq p,$ we have $z-p+j-1 \geq 0.$ Therefore,
	\begin{align*}
		\beta_j+n 
		= g_{z-p+j-1} - \epsilon m_0 +(\nu - 1)m_0
		&= (\nu - 1)m_0- \epsilon m_0 + m_{r_{z+j-p-1}}+q_{z-p+j-1}m_p,
	\end{align*}

	with $\epsilon =1$ or $0$ according as $r'+ r_{z-p+j-1} \leq p$ or  $r'+ r_{z-p+j-1} > p.$ \\
	
	(b) If $z < p,$ then $z = r_z = p+r-r' = p+1-r'.$ Also in this case $\PF_2(H) = \{ \beta_j \mid j \in [r',p] \}$ (see {\cite[Proposition 5.7]{patil-sengupta}}). Since $j \geq r'=p-z+1,$ we get $z-p+j-1 \geq 0.$ Hence we are done by subcase(a).
\end{proof}

\begin{proposition}\label{r=1,L=1,PF_2(H)}
	Let $r=1$ and $\lambda =1$. Then for each  $\beta_j \in \PF_2(H)$, $j \in [1,p]$, there exists an $\RF(\beta_j)$ of the following type
	\begin{multline*}
		\RF(\beta_j)
		=
		\begin{bmatrix}
			-1 & 0 & 0 & \cdots& \cdots & \cdots & 0 & q-1 & v-w-1\\
			0 & -1 & 0 & \cdots & \cdots & \cdots & 0 & q-1 & v-w-1\\
			\vdots & \vdots & \vdots & \vdots & \vdots & \vdots  & \vdots  & \vdots  & \vdots \\
			0 & 0 & \cdots & -1 & 0 & \cdots & 0 & q-1 & v-w-1\\
			\lambda -1 & 0 & 0 & \cdots & -1 & 0 & \cdots & 0 & v-1\\
			\nu-2 & 0 & 0 & \cdots & \cdots & \cdots & 0 & -1 & w-1\\
			\vdots & \vdots & \vdots & \vdots & \vdots & \vdots  & \vdots  & \vdots  & \vdots \\
			\nu-2 & 0 & 0 & \cdots & \cdots & \cdots & 0 & -1 & w-1\\
			\nu-\epsilon & 0 & \cdots & \cdots & \cdots &\cdots & 0 & q_{z+j-p-1} & -1\\
		\end{bmatrix}
		+
		\begin{bmatrix}
			R_{(0)} \\ R_{(1)} \\ \vdots \\  R_{(p-j)} \\ 0 \\ R_{(p-j+2)}\\ \vdots \\  R_{(p)} \\ R_{(n)}\\
		\end{bmatrix},
	\end{multline*}
	where $\epsilon =2$ or $1$ according as $r'+ r_{z-p+j-1} \leq p$ or  $r'+ r_{z-p+j-1} > p$ and
	\begin{enumerate}
		\item $R_{(0)} = (a^{(0)},\ldots,a^{(j)},\ldots,a^{(n)})$ such that $a^{(j)} = 1$ and all other $a^{(k)}$'s equal to zero,
		
		\item for all $k \in [1,p-j]$, $R_{(k)} = (a^{(0)},\ldots,a^{(j+k)},\ldots,a^{(n)})$ such that $a^{(j+k)} = 1$ and all other $a^{(k)}$'s equal to zero,
		
		\item for all $k \in [p-j+2,p]$, $R_{(k)} = (a^{(0)},\ldots,a^{(k+j-p-2)},\ldots,a^{(n)})$ such that $a^{(k+j-p-2)} = 1$ and all other $a^{(k)}$'s equal to zero,
		
		\item $R_{(n)} = (a^{(0)},\ldots,a^{(r_{z+j-p-1})},\ldots,a^{(n)})$ such that $a^{(r_{z+j-p-1})} = 1$ and all other $a^{(k)}$'s equal to zero.
	\end{enumerate}
\end{proposition}
              
\begin{proof}
	Same as Proposition \ref{r=1,L>1,PF_2(H)} except when $j+k \geq p+2.$\\
	If $j+k = p+2,$ then $\beta_j + m_k = (\nu -1)m_0 + (w-1)n.$ \\
	If $j+k > p+2,$ then $\beta_j + m_k = (\nu -2)m_0 + g_{j+k-p-2} + (w-1)n.$
	Hence the result follows.
\end{proof}

Note that if $p=1,$ then $\PF_2(H)=\{\beta_1\}.$

\begin{coro}\label{p=1,PF_2(H)}
	Suppose $p=1$. Then there exists an $\RF$-matrix of $\beta_1$ of the following type
	\[
	\begin{bmatrix}
		-1 &  u-1 & v-w-1     \\
		\lambda-1 &  -1 & v-1    \\
		\nu-1 & z-1 & -1 
	\end{bmatrix}.
	\]
\end{coro}

\begin{proof}
	Use Propositions \ref{r=1,L>1,PF_2(H)} and \ref{r=1,L=1,PF_2(H)}.
\end{proof}

Let $H$ be a numerical semigroup generated by a sequence $m_0,\ldots,m_e$. Let $k$ be a field. The semigroup ring $k[H]$ of $H$ is a $k$-subalgebra of the polynomial ring $k[t].$ In other words, $k[H] = k[t^{m_0},t^{m_1},\ldots,t^{m_e}].$ Set $R = k[x_0,\ldots,x_{e}]$ and define a map $\pi : R \rightarrow k[H]$ given by $\pi(x_i) = t^{m_{i}}$ for all $i=0,\ldots,e.$ Set $\deg x_i = m_{i}$ for all $i=0,\ldots,e.$ Observe that $R$ is a graded ring and that $\pi$ is a degree preserving surjective $k$-algebra homomorphism. We say kernel of $\pi$ is defining ideal of $k[H],$ denoted by $I(H)$. 

For a vector $b = (b_1,\ldots,b_n) \in \mathbb{Z}^n$, we let $b^{+}$ denote the vector whose $i$-th entry is $b_i$ if $b_i > 0$, and is zero otherwise, and we let $b^{-} = b^{+} - b.$ Then $b = b^{+} - b^{-}$ with $b^{+}, b^{-} \in \mathbb{N}^n.$ For a pseudo-Frobenius element $f$, the rows of $\RF(f)$ produce binomials in $I(H).$

\begin{definition}\label{RF-relation}
{\rm For some pseudo-Frobenius element $f$ of $H$, let $\delta_1, \ldots, \delta_n$ denote 
the row vectors of $\RF(f).$ Set $\delta_{(ij)} = \delta_j - \delta_i$, for all 
$1 \leq i < j \leq n.$ Then $\phi_{ij} = {\bf x}^{\delta_{(ij)}^+} - {\bf x}^{\delta_{(ij)}^-} \in I(H)$ 
for all $i < j.$ We call $\phi_{ij}$ an $\RF(f)$-relation. We call a binomial relation $\phi \in I(H)$ 
an $\bf{\RF}$-relation if it is an $\RF(f)$-relation for some $f \in \PF (H)$.
}
\end{definition}

\begin{proposition}\label{I(H), RF-relation}
	If $p=1$ and $\mu > 0,$ then $I(H)$ is minimally generated by $\RF(\beta_1)$-relations.
\end{proposition}

\begin{proof}
	Since $p=1$ and $\mu > 0,$ using {\cite[Propositions 4.6 and 5.6]{patil-sengupta}} it follows that $\PF_1(H) = \{\alpha_1\}$ and $\PF_2(H) = \{\beta_1\}.$ Therefore, we get $\PF(H) = \{\alpha_1, \beta_1\}$ and $H$ is not symmetric. Also, as $p=1,$ we have $q_t = t-1$ for all $t.$ In \cite[Theorem 4.5]{patil93}, it is proved that in the above setup, 
	\[ I(H) = \langle x_0^{\nu} - x_1^{q'+1} x_2^{v-w} , x_1^{q+1} - x_0^{\lambda} x_2^w, x_2^v - x_0^{\mu} x_1^{q-q'} \rangle. \]
	Consider the matrix
	\[ \RF(\beta_1):=
	\begin{bmatrix}
	    -1 &  u-1 & v-w-1     \\
	    \lambda-1 &  -1 & v-1    \\
	    \nu-1 & z-1 & -1 
	\end{bmatrix}.
	\]
	If $\delta_1,\delta_2,\delta_3$ represent the row vectors of $\RF(\beta_1)$, then we have 
	the following $\RF(\beta_1)$-relations:
	$\delta_{12} = x_1^u - x_0^\lambda x_2^w, \delta_{13}=x_0^{\nu} - x_1^{u-z}x_2^{v-w}$ and $\delta_{23}= x_2^v - x_0^\mu x_1^z.$ Hence the result.
\end{proof}

\begin{proposition}\label{r geq 2, PF_2(H)}
	Let $r \geq 2$. Then for each  $\beta_j \in \PF_2(H)$, $j \in [p-r+2,p]$, there exists an $\RF(\beta_j)$ of the following type
	\begin{multline*}
		\RF(\beta_j)
		=
		\begin{bmatrix}
			-1 & 0 & 0 & \cdots& \cdots & \cdots & 0 & q & v-w-1\\
			0 & -1 & 0 & \cdots & \cdots & \cdots & 0 & q & v-w-1\\
			\vdots & \vdots & \vdots & \vdots & \vdots & \vdots  & \vdots  & \vdots  & \vdots \\
			0 & 0 & \cdots & -1 & 0 & \cdots & 0 & q & v-w-1\\
			\lambda -1 & 0 & 0 & \cdots & -1 & 0 & \cdots & 0 & v-1\\
			\vdots & \vdots & \vdots & \vdots & \vdots & \vdots  & \vdots  & \vdots  & \vdots \\
			\lambda-1 & 0 & 0 & \cdots & \cdots & \cdots & 0 & -1 & v-1\\
			\nu-\epsilon & 0 & \cdots & \cdots & \cdots &\cdots & 0 & q_{z+j-p-1} & -1\\
		\end{bmatrix}
		+
		\begin{bmatrix}
			R_{(0)} \\ R_{(1)} \\ \vdots \\  R_{(e)} \\ R_{(e+1)}\\ \vdots \\  R_{(p)} \\ R_{(n)}\\
		\end{bmatrix},
	\end{multline*}
	where $e \in [1,p-1]$ such that $r+j+e = 2p+1$ and $\epsilon =2$ or $1$ according as $r'+ r_{z-p+j-1} \leq p$ or  $r'+ r_{z-p+j-1} > p,$ and 
	\begin{enumerate}
		\item $R_{(0)} = (a^{(0)},\ldots,a^{(p-e)},\ldots,a^{(n)})$ such that $a^{(p-e)} = 1$ and all other $a^{(k)}$'s equal to zero,
		\item for all $k \in [1,e]$, $R_{(k)} = (a^{(0)},\ldots,a^{(p-e+k)},\ldots,a^{(n)})$ such that $a^{(p-e+k)} = 1$ and all other $a^{(k)}$'s equal to zero,
		\item for all $k \in [e+1,p]$, $R_{(k)} = (a^{(0)},\ldots,a^{(j+k-p-1)},\ldots,a^{(n)})$ such that $a^{(j+k-p-1)} = 1$ and all other $a^{(k)}$'s equal to zero,
		\item $R_{(n)} = (a^{(0)},\ldots,a^{(r_{z+j-p-1})},\ldots,a^{(n)})$ such that $a^{(r_{z+j-p-1})} = 1$ and all other $a^{(k)}$'s equal to zero.
	\end{enumerate}
\end{proposition}

\begin{proof}
	Use similar arguments as in Propositions \ref{r geq 2} and \ref{r=1,L>1,PF_2(H)}.
\end{proof}

\section{$\RF$-relations in the symmetric case}

In this section, we consider those numerical semigroups generated by an almost arithmetic sequence which are symmetric. In particular, let $H = \langle m_0,m_1,\ldots,m_p,n \rangle$, where $m_i = m_0 + id$ for $i \in [1,p]$ and $\gcd(m_0,n,d)=1$ such that $\text{type}(H) = |\PF(H)| = 1.$ Using the language from the previous sections, we list the different cases in which $\text{type}(H)=1$ and also give a possible $\RF$-matrix of the Frobenius number in each case. The aim of this section is to prove that $I(H)$ has a minimal generating set consisting of $\RF$-relations in the above setup if the embedding dimension is either 4 or 5.\\

Let $H$ be a numerical semigroups generated by an almost arithmetic sequence. Then $\text{type}(H)=1$ in the following cases: \\

\textbf{Case 1.} Let $W = \emptyset.$ In this case, if $r=1$, then $H$ is of type 1 only if $p=1$ and $\PF(H)=\{\gamma_1\}$. If $r \geq 2,$ then $H$ is of type 1 only if $r=2$ and $\PF(H)=\{\gamma_p\}.$ \\

\textbf{Case 2.} Let $W \neq \emptyset.$ If $p=1,$ then $H$ is of type 1 only if $\mu =0$ and $\PF(H)=\{\beta_1\}.$ \\
Suppose $p \geq 2.$ The only cases in which $H$ is of type 1 are the following:
\begin{enumerate}
	\item[(i)] $r=1,$ $r'=2,$ $\lambda=1,$ $z<p$, and $\PF(H) = \{\alpha_{p+1}\}.$ 
	\item[(ii)] $r=2,$ $r'=1,$ $\mu =0,$ $q'=0$, and $\PF(H)=\{\beta_p\}.$
	\item[(iii)] $r=2,$ $r'=2,$ $\mu=0$, and $\PF(H)=\{\beta_p\}.$ 	 
\end{enumerate}

\noindent
We now consider each case separately and list an $\RF$-matrix of the Frobenius number in each case.\\

\textbf{Case 1.} Let $W = \emptyset.$\\

Let $p=1,$ $r=1$ and $\PF(H)=\{\gamma_1\}$. Then we have choices for the $\RF(\gamma_1)$-matrix as given in Corollary \ref{p=1}.

Let $r=2$ and $p \geq 2.$ Then there exists an $\RF$-matrix of $\gamma_p$ as given in Corollary \ref{r=2,ne}. \\

\textbf{Case 2.} Let $W \neq \emptyset.$ \\

Suppose $p=1$ and $\mu =0.$ Then there exists an $\RF$-matrix of $\beta_1$ as given in Corollary \ref{p=1,PF_2(H)}.

Let $p \geq 2.$ We have the following three sub-cases: \\

\textbf{Subcase (i).}
Let $r=1,$ $r'=2$, $\lambda=1$ and $z<p$ and $\PF(H)=\{\alpha_{p+1}\}.$ As a consequence, $r_z = p-1,$ implying that $z = p-1$ and $q_z = 0$. Also, $\lambda=1$ and $r' > r$ implies $\mu \neq 0.$ Using Proposition \ref{r' geq 2}, there exists an $\RF$-matrix of $\alpha_{p+1}$ of the following type
\begin{align*}
\begin{bmatrix}
-1 & 1 & 0 & 0 & \cdots & 0 & 0 & q' & v-1 \\
0 & -1 & 1 & 0 & \cdots & 0 & 0 & q' & v-1 \\
\vdots & \vdots  & \vdots & \vdots & \vdots & \vdots & \vdots & \vdots & \vdots \\
\vdots & \vdots  & \vdots & \vdots & \vdots & \vdots & \vdots & \vdots & \vdots \\
0 & 0 & 0 & \cdots & \cdots & 0 & -1 & q'+1 & v-1 \\
\nu-1 & 0 & 0 & \cdots & \cdots & 0 & 1 & -1 & w-1 \\
\mu-1 & 1 & 0 & \cdots & \cdots & 0 & 1 & q' & -1
\end{bmatrix}.
\end{align*}

\textbf{Subcase (ii).}
Let $r=2,$ $r'=1$, $\mu =0$, $q'=0$ and $\PF(H) = \{\beta_p\}.$ Then $r_z = 1$ and hence $r_{z-1} = p$ and $q_{z-1} = q_z - 1.$ Since $q'=0,$ we get $u-z = r'=1$ which implies $z=u-1=qp+1$ and hence $q_z = q.$ Also $r' + r_{z-1} = p+1 > p$ implies $\epsilon = 1$ and as $\mu = 0$ and $r' < r$, we get $\nu = \lambda + 1.$ By Proposition \ref{r geq 2, PF_2(H)}, there exists an $\RF$-matrix of $\beta_p$ of the following type 
\begin{align*}
\begin{bmatrix}
-1 & 1 & 0 & 0 & \cdots & 0 & 0 & q & v-w-1 \\
0 & -1 & 1 & 0 & \cdots & 0 & 0 & q & v-w-1 \\
\vdots & \vdots  & \vdots & \vdots & \vdots & \vdots & \vdots & \vdots & \vdots \\
\vdots & \vdots  & \vdots & \vdots & \vdots & \vdots & \vdots & \vdots & \vdots \\
0 & 0 & 0 & \cdots & \cdots & 0 & -1 & q+1 & v-w-1 \\
\lambda-1 & 0 & 0 & \cdots & \cdots & 0 & 1 & -1 & v-1 \\
\lambda & 0 & 0 & \cdots & \cdots & 0 & 0 & q & -1
\end{bmatrix}.
\end{align*}

\textbf{Subcase (iii).}
Let $r=2$, $r'=2,$ $\mu=0$ and $\PF(H) = \{\beta_p\}.$ Then $r_z = p$ and hence $r_{z-1} = p-1$ and $q_{z-1} = q_z .$ Also, $r' + r_{z-1} = p+1 > p$ implies $\epsilon = 1$ and as $\mu = 0$ and $r=r'$, we get $\nu = \lambda.$ By Proposition \ref{r geq 2, PF_2(H)}, there exists an $\RF$-matrix of $\beta_p$ of the following type
\begin{align*}
\begin{bmatrix}
-1 & 1 & 0 & 0 & \cdots & 0 & 0 & q & v-w-1 \\
0 & -1 & 1 & 0 & \cdots & 0 & 0 & q & v-w-1 \\
\vdots & \vdots  & \vdots & \vdots & \vdots & \vdots & \vdots & \vdots & \vdots \\
\vdots & \vdots  & \vdots & \vdots & \vdots & \vdots & \vdots & \vdots & \vdots \\
0 & 0 & 0 & \cdots & \cdots & 0 & -1 & q+1 & v-w-1 \\
\lambda-1 & 0 & 0 & \cdots & \cdots & 0 & 1 & -1 & v-1 \\
\lambda-1 & 0 & 0 & \cdots & \cdots & 0 & 1 & q_z & -1
\end{bmatrix}.
\end{align*}
\medskip

For a pseudo-Frobenius element $f \in H,$ the difference of any two rows of an $\RF(f)$-matrix produces a binomial, which we call an $\RF(f)$-relation (see Definition \ref{RF-relation}). It is easy to check that these binomials are elements in $I(H).$ In \cite{herzogWatanabe}, the authors raised the following question: When is $I(H)$ minimally generated by $\RF$-relations? The authors answer the question in the affirmative when the embedding dimension is 3 and also when the embedding dimension is 4 and the semigroup is pseudo-symmetric or almost symmetric. They also remark that in the case of symmetric semigroups with embedding dimension 4, $\RF(H)$ cannot be determined uniquely and that all the generators of $I(H)$ cannot be obtained from a single $\RF(H).$

When $H$ is generated by an almost arithmetic sequence, using \cite[Theorem 1]{etoGeneric} and the structure of $\RF$-matrices (as given in section 3 and 4), it is easy to check that the $\RF$-matrices are not unique. We now answer the question of Herzog-Watanabe in the affirmative for the symmetric numerical semigroups generated by an almost arithmetic sequence, with embedding dimension 4 or 5.

\begin{proposition}\label{W = emptyset,p=2,RF-relation}
	Suppose $W = \emptyset$ and $p=2$. If $H$ is symmetric, then $I(H)$ is minimally generated by $\RF$-relations.
\end{proposition}

\begin{proof}
	Since $W = \emptyset$, $p=2$ and $H$ is symmetric, we get $r=2$ and $\PF(H) = \{\gamma_2\}$. Let $\delta_1, \delta_2, \delta_3$ and $\delta_4$ denote the rows of $\RF(\gamma_2)$ where
	\begin{align*}
		\RF(\gamma_2) = 
		\begin{bmatrix}
			-1 &  1 & q & v-1     \\
			0 & -1 &  q+1 & v-1    \\
			\lambda-1 & 1 & -1 & w+v-1  \\
			a_{41} & a_{42} & a_{43} & a_{44}
		\end{bmatrix},
	\end{align*}
	and $\delta_4 = (a_{41}, a_{42}, a_{43}, a_{44})$ is as follows:
	\begin{enumerate}
		\item If $r'=1$ and $\mu >0,$ then $\delta_4 = (\mu-1, 2, q+q_z,-1).$
		\item If $r'=2$ and $\mu >0,$ then $\delta_4 = (\mu-1, 1, q+q_z+1,-1).$
		\item If $r'=1$ and $\mu=0,$ then $\delta_4 = (\lambda-1, 2, q_z-1,-1).$
		\item If $r'=2$ and $\mu=0,$ then $\delta_4 = (\lambda-1, 1, q_z,-1).$
	\end{enumerate}
	In \cite[Theorem 4.5]{patil93}, the author gives a minimal generating set of the defining ideal $I(H).$ It is therefore sufficient to show that these binomials arise from $\RF(\gamma_2)$-relations.  \\
	
	(1) Suppose $\mu > 0$ and $r'=1$. Since $r=2$ and $r'= 1$, we get $r_z = 1$ and $q = q' + q_z.$ Then by \cite[Theorem 4.5]{patil93}, the set $\{x_0x_2 - x_1^2, x_0^{\lambda}x_3^w - x_2^{q+1}, x_0^{\mu}x_1x_2^{q_z}-x_3^v \}$ forms a minimal generating set for $I(H).$ These binomials arise from the $\RF(\gamma_2)$-relations corresponding to $\delta_{12} = (1,-2,1,0)$, $\delta_{13} = (\lambda,0,-(q+1),w)$, and $\delta_{14} = (\mu, 1 , q_z, -v)$ respectively. \\ 
	
	(2) Suppose $\mu > 0$ and $r'= 2$. Since $r=2$ and $r'= 2$, we get $r_z = 2$ and $q=q'+q_z+1.$ Then by {\cite[Theorem 4.5]{patil93}}, the set $\{x_0x_2 - x_1^2, x_0^{\lambda}x_3^w - x_2^{q+1}, x_0^{\mu}x_2^{q_z+1}-x_3^v\}$ forms a minimal generating set for $I(H).$ These binomials arise from the $\RF(\gamma_2)$-relations corresponding to $\delta_{12} = (1,-2,1,0)$, $\delta_{13} = (\lambda,0,-(q+1),w)$, and $\delta_{14} = (\mu, 0 , q_z + 1, -v)$ respectively. \\
	
	(3) Suppose $\mu = 0$ and $r'=1$. Since $\mu = 0$, we have $z > p$ and hence $w=0.$ Also, $r'=1$ implies $r_z = 1$ and $q=q'+q_z+1.$ Then by {\cite[Theorem 4.5]{patil93}}, the set $\{x_0x_2 - x_1^2, x_0^{\lambda} - x_2^{q+1}, x_1x_2^{q_{z}+1}-x_3^v \}$ forms a minimal generating set for $I(H).$ These binomials arise from the $\RF(\gamma_2)$-relations corresponding to $\delta_{12} = (1,-2,1,0)$, $\delta_{13} = (\lambda,0,-(q+1),0)$, and $\delta_{34} = (0, 1 , q_z + 1, -v)$ respectively.\\
	
	(4) Suppose $\mu = 0$ and $r' = 2$. Since $r=2$ and $r'=2$, we get $r_z = 2$ and $q=q'+q_z+1.$ Then by \cite[Theorem 4.5]{patil93}, the set $\{x_0x_2 - x_1^2, x_0^{\lambda} - x_2^{q+1}, x_2^{q_{z}+1}-x_3^v \}$ forms a minimal generating set for $I(H).$ These binomials arise from the $\RF(\gamma_2)$-relations corresponding to $\delta_{12} = (1,-2,1,0)$, $\delta_{13} = (\lambda,0,-(q+1),0)$, $\delta_{34} = (0, 0 , q_z + 1, -v)$ respectively. 
\end{proof}

\begin{proposition}\label{W neq emptyset, p=2, RF-relations}
	Suppose $W \neq \emptyset$ and $p=2$. If $H$ is symmetric, then $I(H)$ is minimally generated by  $\RF$-relations.
\end{proposition}

\begin{proof}
	We have the following sub-cases.\\
	
	\textbf{Subcase 1.} Suppose $r=1$, $r'=2$, $\lambda=1$, $z < 2$ and $\PF(H)= \{\alpha_3\}$.\\
	Since $W \neq \emptyset$ and $z < 2(=p)$, we get $\mu \neq 0$ and $z =1$. Hence we have an $\RF$-matrix for $\alpha_3$,
	\begin{align*}
		\RF(\alpha_3)=
		\begin{bmatrix}
			-1 &  1 & q' & v-1     \\
			0 & -1 &  q'+1 & v-1    \\
			\nu-1 & 1 & -1 & w-1  \\
			\mu-1 & 2 & q' & -1
		\end{bmatrix}.
	\end{align*}
	Since $(\mu-1)m_0+2m_1+q'm_2 = (\nu-2)m_0 + 2m_1 +(q-1)m_2 = (\nu-1)m_0 + qm_2$, we get another $\RF$-matrix for $\alpha_3$,
	\begin{align*}
		\RF(\alpha_3)'=
		\begin{bmatrix}
			-1 &  1 & q' & v-1     \\
			0 & -1 &  q'+1 & v-1    \\
			\nu-1 & 1 & -1 & w-1  \\
			\nu-1 & 0 & q & -1
		\end{bmatrix}.
	\end{align*}
	Let $\delta_1,\delta_2,\delta_3,\delta_4$ and $\delta_1',\delta_2',\delta_3',\delta_4'$ denote the rows of $\RF(\alpha_3)$ and $\RF(\alpha_3)'$ respectively. By \cite[Theorem 4.5]{patil93}, the set 
	\[ \{ x_0x_2 - x_1^2, \ x_0^{\nu} - x_2^{q'+1}x_3^{v-w}, \ x_0^{\mu}x_1 - x_3^v , \ x_1x_2^{q'+1}-x_0x_3^w, \ x_2^{q+1}-x_1x_3^w \}\] 
	forms a minimal generating set for $I(H).$ These binomials arise from the $\RF(\alpha_3)$-relations corresponding to $\delta_{12} = (1,-2,1,0)$, $\delta_{13} =  (\nu,0,-(q'+1),-(v-w))$, $\delta_{14} = (\mu,1,0,-v)$, $\delta_{34} = (-1, 1 , q'+1, -w)$ and $\RF(\alpha_3)'$-relation corresponding to $\delta_{34}'= (0, -1, q+1, -w)$ respectively. \\
	
	\textbf{Subcase 2.} Suppose $r' = 1$, $r = 2$, $\mu = 0$, $q'= 0$ and $\PF(H) = \{\beta_2\}.$\\
	We have an $\RF$-matrix for $\beta_2$,
	\begin{align*}
		\RF(\beta_2)=
		\begin{bmatrix}
			-1 &  1 & q & v-w-1     \\
			0 & -1 &  q+1 & v-w-1    \\
			\lambda-1 & 1 & -1 & v-1  \\
			\lambda & 0 & q & -1
		\end{bmatrix}.
	\end{align*}
	Since $(q+1)m_2 + (v-w-1)n = \lambda m_0 + (v-1)n$, we get another $\RF$-matrix for $\beta_2$,
	\begin{align*}
		\RF(\beta_2)' =
		\begin{bmatrix}
			-1 &  1 & q & v-w-1     \\
			\lambda & -1 &  0 & v-1    \\
			\lambda-1 & 1 & -1 & v-1  \\
			\lambda & 0 & q & -1
		\end{bmatrix}.
	\end{align*}
	Let $\delta_1,\delta_2,\delta_3,\delta_4$ and $\delta_1',\delta_2',\delta_3',\delta_4'$ denote the rows of $\RF(\beta_2)$ and $\RF(\beta_2)'$ respectively. By \cite[Theorem 4.5]{patil93}, the set 
	\[ \{ x_0x_2 - x_1^2, \ x_0^{\lambda}x_3^w - x_2^{q+1}, \ x_0^{\lambda+1} -x_1x_3^{v-w}, \ x_0^{\lambda}x_1-x_2x_3^{v-w}, \ x_1x_2^q-x_3^v \} \] 
	forms a minimal generating set for $I(H).$ These binomials arise from the $\RF(\beta_2)$-relations corresponding to $\delta_{12} = (1,-2,1,0)$, $\delta_{13} =  (\lambda,0,-(q+1),w)$, $\delta_{14} = (\lambda+1,-1,0,-(v-w))$, $\delta_{24} = (\lambda, 1 ,-1, -(v-w))$ and $\RF(\beta_2)'$-relation corresponding to $\delta_{24}' = (0, 1, q, -v)$ respectively. \\
	
	\textbf{Subcase 3.} Suppose $r = 2$, $r' = 2$, $\mu = 0$ and $\PF(H) = \{\beta_2\}.$ \\
	We have an $\RF$-matrix for $\beta_2,$
	\[
	\RF(\beta_2)=
	\begin{bmatrix}
		-1 &  1 & q & v-w-1     \\
		0 & -1 &  q+1 & v-w-1    \\
		\lambda-1 & 1 & -1 & v-1  \\
		\lambda-1 & 1 & q_z & -1
	\end{bmatrix}.
	\]
	Let $\delta_1,\delta_2,\delta_3$ and $\delta_4$ denote the rows of $\RF(\beta_2)$. By \cite[Theorem 4.5]{patil93}, the set 
	\[ \{ x_0x_2 - x_1^2, \ x_0^{\lambda} - x_2^{q-q_z}x_3^{v-w}, \ x_2^{q_z+1}-x_3^v\} \]
	forms a minimal generating set for $I(H).$ These binomials arise from the $\RF(\beta_2)$-relations corresponding to $\delta_{12} = (1,-2,1,0)$, $\delta_{14} = (\lambda,0,-(q-q_z),-(v-w))$, and $\delta_{34} = (0, 0 , q_z + 1, -v)$ respectively. 
\end{proof}

\begin{theorem}   \label{p=2,t=1}
	Let $H = \langle m_0, m_1,m_2,n \rangle$ be a symmetric numerical semigroup generated by an almost arithmetic sequence. Then $I(H)$ has a minimal generating set consisting of $\RF$-relations.
\end{theorem}

\begin{proof}
	Combine Propositions \ref{W = emptyset,p=2,RF-relation} and \ref{W neq emptyset, p=2, RF-relations}.
\end{proof}

\begin{remark}
{\rm	If $H$ is a numerical semigroup with embedding dimension 4, then $H$ is a complete intersection if $\mu(I(H))=3.$ In \cite{bresinsky3}, Bresinsky proves that $\mu(I(H))$ for a symmetric numerical semigroup $H$ with embedding dimension four is either three or five. From the proofs of Propositions \ref{W = emptyset,p=2,RF-relation} and \ref{W neq emptyset, p=2, RF-relations}, it is easy to check that a symmetric numerical semigroup $H$ generated by an almost arithmetic sequence is a complete intersection if either \\
	(1) $W = \emptyset,$ $r=2$ and $\PF(H)=\{\gamma_2\},$ or  \\
	(2) $W \neq \emptyset,$ $r=r'=2,$ $\mu=0$ and $\PF(H)=\{\beta_2\}.$ \\
	Also, for a numerical semigroup generated by an almost arithmetic sequence with embedding dimension three or four, a complete description of being complete intersection in terms of above notation is given in \cite[Proposition 3.6]{maloo-sengupta}.
	}
\end{remark}

\begin{theorem}   \label{p=3,t=1}
	Let $H = \langle m_0, m_1,m_2,m_3,n \rangle$ be a symmetric numerical semigroup generated by an almost arithmetic sequence. Then $I(H)$ has a minimal generating set consisting of $\RF$-relations.
\end{theorem}

\begin{proof}
	We use the minimal generating set of $I(H)$ as given in \cite[Theorem 4.5]{patil93}. We first note that the subset of generators 
	\[ \{ \xi_{11}= x_0x_2 - x_1^2, \ \xi_{12}= x_0x_3 - x_1x_2, \ \xi_{22}= x_1x_3 - x_2^2 \} \] 
	appears in every case below and that these binomials arise from the $\RF$-relations corresponding to $\delta_{12},$ $\delta_{13}$ and $\delta_{23}$ respectively, where $\delta_1$, $\delta_2$ and $\delta_3$ are rows of the respective $\RF$-matrices in each case. \\
	 
	{\bf Case 1.} Suppose $W = \emptyset.$ Then $p=3,$ $r=2$ and $\PF(H) = \{\gamma_3 \}.$ \\
	By \cite[Theorem 4.5]{patil93}, it follows that $I(H)$ is minimally generated by
	\begin{align*}
		\{ \xi_{11}, \ \xi_{12}, \ \xi_{22}, \ \phi_0 = x_0^\lambda x_4^w - x_2x_3^q, \ \phi_1 = x_0^{\lambda-1}x_1x_4^w - x_3^{q+1}, \ \theta \},
	\end{align*}
	where $\theta = x_0^{\mu}x_1x_3^{q-q'} - x_4^v$ when $r'=1$ and $\theta = x_0^{\mu}x_{5-r'}x_3^{q-q'-1} - x_4^v$ when $r'=2,3.$ Also, by Corollary \ref{r=2,ne}, the matrix
	\[ \RF(\gamma_3) =
	\left[
	\begin{array}{c c c c c}
		-1 & 1 & 0 & q & v-1 \\
		0 & -1 & 1 & q & v-1 \\
		0 & 0 & -1 & q+1 &v-1 \\
		\lambda -1 & 0 & 1 & -1 & w+v-1 \\
		a_{5,1} & a_{5,2} & a_{5,3} & a_{5,4} & -1
	\end{array}
	\right],
	\]
	where the last row depends on the value of $\mu.$ We observe that the binomial $\phi_1$ arise from the $\RF$-relations corresponding to $\delta_{24}.$ Since $m_1 + qm_3 = 2m_2 + (q-1)m_3$, we get another $\RF$-matrix for $\gamma_3$,
	\[ \RF(\gamma_3)' =
	\left[
	\begin{array}{c c c c c}
		-1 & 0 & 2 & q-1 & v-1 \\
		0 & -1 & 1 & q & v-1 \\
		0 & 0 & -1 & q+1 &v-1 \\
		\lambda -1 & 0 & 1 & -1 & w+v-1 \\
		a_{5,1} & a_{5,2} & a_{5,3} & a_{5,4} & -1
	\end{array}
	\right].
	\]
	The binomial $\phi_0$ arises from the $\RF(\gamma_3)'$-relation corresponding to $\delta_{14}.$ In order to prove that $\theta$ arises from an $\RF$-relation, we consider the following cases:
	\begin{enumerate}
		\item Let $\mu=0$ and $r'=1.$ Then $\delta_5 = (\lambda-1, 1, 1, q_z - 1, -1)$, $w=0$ and $\theta$ arises from the $\RF$-relation corresponding to $\delta_{45}.$
		
		\item Let $\mu=0$ and $r'=2.$ Then $\delta_5 = (\lambda-1, 0, 1, q_z, -1)$, $w=0$ and $\theta$ arises from the $\RF$-relation corresponding to $\delta_{45}.$
		
		\item Let $\mu=0$ and $r'=3.$ Then $\delta_5 = (\lambda-1, 0, 2, q_z - 1, -1)$, $w=0$ and $\theta$ arises from the $\RF$-relation corresponding to $\delta_{45}.$
		
		\item Let $\mu>0$ and $r'=1.$ Then $\delta_5 = (\mu-1, 2, 0, q+q_z, -1)$ and $\theta$ arises from the $\RF$-relation corresponding to $\delta_{15}.$
		
		\item Let $\mu>0$ and $r'=2.$ Then $\delta_5 = (\mu-1, 1, 0, q+q_z+1, -1)$ and $\theta$ arises from the $\RF$-relation corresponding to $\delta_{15}.$
		
		\item Let $\mu>0$ and $r'=3.$ Then $\delta_5 = (\mu-1, 1, 1, q+q_z, -1)$ and $\theta$ arises from the $\RF$-relation corresponding to $\delta_{15}.$     
	\end{enumerate}
\medskip

	{\bf Case 2.} Suppose $W \neq \emptyset.$ We have the following sub-cases.\\
	
	\textbf{Subcase 1.} Suppose $r=1$, $r'=2$, $\lambda=1$, $z < 3$ and $\PF(H)= \{\alpha_4\}$.\\
	We consider the following $\RF$-matrices for $\alpha_4$,
	\begin{align*}
		\RF(\alpha_4)=
		\begin{bmatrix}
			-1 &  1 & 0 & q' & v-1     \\
			0 & -1 & 1 &  q' & v-1    \\
			0 & 0 & -1 & q'+1 & v-1   \\
			\mu & 0 & 1 & -1 & w-1  \\
			\mu-1 & 1 & 1 & q' & -1
		\end{bmatrix}
		, \
		\RF(\alpha_4)' =
		\begin{bmatrix}
			-1 &  1 & 0 & q' & v-1     \\
			0 & -1 & 1 &  q' & v-1    \\
			0 & 0 & -1 & q'+1 & v-1   \\
			\mu-1 & 2 & 0 & -1 & w-1  \\
			\mu-1 & 1 & 1 & q' & -1
		\end{bmatrix},
	\end{align*}
	and
	\begin{align*}
		\RF(\alpha_4)''=
		\begin{bmatrix}
			-1 & 0 & 2 & q'-1 & v-1     \\
			0 & -1 & 1 &  q' & v-1    \\
			0 & 0 & -1 & q'+1 & v-1   \\
			\mu & 0 & 1 & -1 & w-1  \\
			\mu & 0 & 0 & q'+1 & -1
		\end{bmatrix}.
	\end{align*}
	For $i=1,\ldots,5,$ let $\delta_i$, $\delta'_i$ and $\delta''_i$ denote the $i$-th rows of $\RF(\alpha_4)$, $\RF(\alpha_4)'$ and $\RF(\alpha_4)''$ respectively. By \cite[Theorem 4.5]{patil93}, the set 
	\begin{align*}
		\{ \xi_{11}, \ \xi_{12}, \ \xi_{22} \} \bigcup
		\bigg\{ 
		\begin{array}{lll}
			\theta = x_0^{\mu}x_2 - x_4^v, & \psi_0 = x_0^{\mu+1} - x_2x_3^{q'}x_4^{v-w}, & \psi_1 = x_0^{\mu}x_1 - x_3^{q'+1}x_4^{v-w}, \\ 
			\phi_0 = x_0x_4^w - x_1x_3^q, & \phi_1 = x_1x_4^w - x_2x_3^q, & \phi_2 = x_2x_4^w - x_3^{q+1} 
		\end{array}
		\bigg\} 
	\end{align*}
	forms a minimal generating set for $I(H).$ The binomials $\theta$, $\phi_0$ and $\psi_1$ arise from the $\RF(\alpha_4)$-relations corresponding to $\delta_{15}$, $\delta_{45},$ and $\delta_{24}$ respectively. The binomials $\phi_1$, $\phi_2$ and $\psi_0$ arise from the $\RF(\alpha_4)'$-relations corresponding to $\delta'_{45}$ and $\RF(\alpha_3)''$-relations corresponding to $\delta_{45}''$ and $\delta''_{14}$ respectively. \\
	
	\textbf{Subcase 2.} Suppose $r' = 1$, $r = 2$, $\mu = 0$, $q'= 0$ and $\PF(H) = \{\beta_3\}.$\\
	We consider the following $\RF$-matrices for $\beta_3$,
	\begin{align*}
		\RF(\beta_3)=
		\begin{bmatrix}
			-1 & 1 & 0 & q & v-w-1    \\
			0 & -1 & 1 & q & v-w-1    \\
			0 & 0 & -1 & q+1 & v-w-1   \\
			\lambda-1 & 0 & 1 & -1 & v-1  \\
			\lambda & 0 & 0 & q & -1
		\end{bmatrix},
		\RF(\beta_3)' =
		\begin{bmatrix}
			-1 & 0 & 2 & q-1 & v-w-1    \\
			0 & -1 & 1 & q & v-w-1    \\
			0 & 0 & -1 & q+1 & v-w-1   \\
			\lambda-1 & 0 & 1 & -1 & v-1  \\
			\lambda-1 & 1 & 1 & q-1 & -1
		\end{bmatrix}.
	\end{align*}
	Let $\delta_1,\delta_2,\delta_3,\delta_4$, $\delta_5$ and $\delta_1',\delta_2',\delta_3',\delta_4'$, $\delta'_5$ denote the rows of $\RF(\beta_3)$ and $\RF(\beta_3)'$ respectively. By \cite[Theorem 4.5]{patil93}, the set 
	\begin{align*}
	\{ \xi_{11}, \ \xi_{12}, \ \xi_{22} \} \bigcup
	\bigg\{ 
	\begin{array}{lll}
		\theta = x_1x_3^{q} - x_4^v, & \phi_0 = x_0^{\lambda}x_4^w - x_2x_3^q, & \phi_1 = x_0^{\lambda-1}x_1x_4^w - x_3^{q+1}, \\ 
		\psi_0 = x_0^{\lambda+1} - x_1x_4^{v-w}, & \psi_1 = x_0^{\lambda}x_1 - x_2x_4^{v-w}, & \psi_2 = x_0^{\lambda}x_2 - x_3x_4^{v-w} 
	\end{array}
	\bigg\} 
	\end{align*}
	forms a minimal generating set for $I(H).$ The binomials $\phi_1$, $\psi_0,$ $\psi_1$ and $\psi_2$ arise from the $\RF(\beta_3)$-relations corresponding to $\delta_{24}$, $\delta_{15}$, $\delta_{25}$, and $\delta_{35}$ respectively. The binomials $\theta$ and $\phi_0$ arise from the $\RF(\beta_3)'$-relations corresponding to $\delta_{45}'$ and $\delta'_{14}$ respectively. \\
	
	\textbf{Subcase 3.} Suppose $r = 2$, $r' = 2$, $\mu = 0$ and $\PF(H) = \{\beta_3\}.$ \\
	We consider the following $\RF$-matrices for $\beta_3,$
	\begin{align*}
	\RF(\beta_3) =
	\begin{bmatrix}
		-1 &  1 & 0 & q & v-w-1   \\
		0 & -1 & 1 &  q & v-w-1    \\
		0 & 0 & -1 & q+1 & v-w-1    \\
		\lambda-1 & 0 & 1 & -1 & v-1  \\
		\lambda-1 & 0 & 1 & q_z & -1
	\end{bmatrix}
	, 
	\RF(\beta_3)' =
		\begin{bmatrix}
			-1 & 0 & 2 & q-1 & v-w-1   \\
			0 & -1 & 1 &  q & v-w-1    \\
			0 & 0 & -1 & q+1 & v-w-1    \\
			\lambda-1 & 0 & 1 & -1 & v-1  \\
			\lambda-1 & 0 & 1 & q_z & -1
		\end{bmatrix}.
	\end{align*}
	Let $\delta_1,\delta_2,\delta_3,$ $\delta_4,$ $\delta_5$ and $\delta'_1, \delta'_2, \delta'_3$, $\delta'_4,$ $\delta'_5$ denote the rows of $\RF(\beta_3)$ and $\RF(\beta_3)'$ respectively. By \cite[Theorem 4.5]{patil93}, the set 
	\[ \{ \xi_{11}, \ \xi_{12}, \ \xi_{22}, \ \theta = x_3^{q_z+1} - x_4^v, \ \psi_0 = x_0^{\lambda} - x_2x_3^{q'}x_4^{v-w}, \ \psi_1 = x_0^{\lambda-1}x_1 - x_3^{q-q_z}x_4^{v-w} \} \]
	forms a minimal generating set for $I(H).$ The binomials $\theta$ and $\psi_1$ arise from the $\RF(\beta_3)$-relations corresponding to $\delta_{45}$ and $\delta_{25}$ respectively and $\psi_0$ arises from the $\RF(\beta_3)'$-relation corresponding to $\delta'_{15}$.  	
\end{proof}

With the knowledge of a minimal generating set of $I(H)$ from \cite{patil93}, one may continue this process in higher embedding dimension as well. The non-generic nature of $I(H)$ and the non-uniqueness of $\RF(H)$ does suggest that the answer to the question is in the affirmative in higher embedding dimension as well. But the same reason turns out to be stumble block when one tries to prove the result in general. 
\medskip

{\it Acknowledgment.} The authors would like to thank the anonymous referee for valuable comments and suggestions.

\bibliographystyle{plain}

\begin{thebibliography}{10}
	
	\bibitem{bresinsky3}
	H.~Bresinsky.
	\newblock Symmetric semigroups of integers generated by {$4$} elements.
	\newblock {\em Manuscripta Math.}, 17(3):205--219, 1975.
	
	\bibitem{gap-ns}
	M.~Delgado, P.~A. Garcia-Sanchez, and J.~Morais.
	\newblock {NumericalSgps}, {A} package for numerical semigroups, {V}ersion
	1.3.0.
	\newblock \href {https://gap-packages.github.io/numericalsgps}
	{\texttt{https://gap-packages.github.io/} \discretionary {}{}{}
		\texttt{numericalsgps}}, Mar 2022.
	\newblock Refereed GAP package.
	
	\bibitem{eto2017}
	Kazufumi Eto.
	\newblock Almost {G}orenstein monomial curves in affine four space.
	\newblock {\em J. Algebra}, 488:362--387, 2017.
	
	\bibitem{etoRowFactor}
	Kazufumi Eto.
	\newblock On row-factorization matrices.
	\newblock {\em Journal of Algebra, Number Theory: Advances and Applications},
	17(2):93--108, 2017.
	
	\bibitem{etoGeneric}
	Kazufumi Eto.
	\newblock {\em Generic Toric Ideals and Row-Factorization Matrices in Numerical
		Semigroups}, pages 83--91.
	\newblock Springer International Publishing, Cham, 2020.

\bibitem{GAP4}
The GAP~Group.
\newblock {\em {GAP -- Groups, Algorithms, and Programming, Version 4.11.1}},
  2021.
	
	\bibitem{herzogWatanabe}
	J\"{u}rgen Herzog and Kei-ichi Watanabe.
	\newblock Almost symmetric numerical semigroups.
	\newblock {\em Semigroup Forum}, 98(3):589--630, 2019.
	
	\bibitem{maloo-sengupta}
	Alok~Kumar Maloo and Indranath Sengupta.
	\newblock Criterion for complete intersection of certain monomial curves.
	\newblock In {\em Advances in algebra and geometry ({H}yderabad, 2001)}, pages
	179--184. Hindustan Book Agency, New Delhi, 2003.
	
	\bibitem{moscariello}
	Alessio Moscariello.
	\newblock On the type of an almost {G}orenstein monomial curve.
	\newblock {\em J. Algebra}, 456:266--277, 2016.
	
	\bibitem{numata}
	T.~Numata.
	\newblock Almost symmetric numerical semigroups generated by four elements.
	\newblock {\em Proc. Inst. Nat. Sci., Nihon Univ.}, 48:197--207, 2013.
	
	\bibitem{patil-sengupta}
	D.~P. Patil and I.~Sengupta.
	\newblock Minimal set of generators for the derivation module of certain
	monomial curves.
	\newblock {\em Comm. Algebra}, 27(11):5619--5631, 1999.
	
	\bibitem{patil93}
	Dilip~P. Patil.
	\newblock Minimal sets of generators for the relation ideals of certain
	monomial curves.
	\newblock {\em Manuscripta Math.}, 80(3):239--248, 1993.
	
	\bibitem{patil-singh}
	Dilip~P. Patil and Balwant Singh.
	\newblock Generators for the derivation modules and the relation ideals of
	certain curves.
	\newblock {\em Manuscripta Math.}, 68(3):327--335, 1990.
	
	\bibitem{RamirezRodseth}
	Jorge~L. Ram\'{\i}rez~Alfons\'{\i}n and {\O}ystein~J. R{\o}dseth.
	\newblock Numerical semigroups: {A}p\'{e}ry sets and {H}ilbert series.
	\newblock {\em Semigroup Forum}, 79(2):323--340, 2009.
	
	\bibitem{rodsethoystein}
	{\O}ystein~J. R{\o}dseth.
	\newblock On a linear {D}iophantine problem of {F}robenius. {II}.
	\newblock {\em J. Reine Angew. Math.}, 307(308):431--440, 1979.
	
	\bibitem{rosales-sanchez}
	J.~C. Rosales and P.~A. Garc\'{\i}a-S\'{a}nchez.
	\newblock {\em Numerical semigroups}, volume~20 of {\em Developments in
		Mathematics}.
	\newblock Springer, New York, 2009.
	
\end{thebibliography}

\end{document}